\documentclass[twoside,11pt,reqno]{amsart}
\usepackage{amsmath,amsthm,amssymb,amstext,amsfonts,amscd}
\usepackage{graphicx}
\usepackage{multirow}

\numberwithin{equation}{section}
\newtheorem{theorem}{Theorem}[section]
\newtheorem{lemma}[theorem]{Lemma}
\newtheorem{definition}[theorem]{Definition}
\newtheorem{proposition}[theorem]{Proposition}

\allowdisplaybreaks

\begin{document}
\title[Growth of solutions of linear differential equations]{Growth
estimates of solutions of linear differential equations with dominant
coefficient of lower $(\alpha ,\beta ,\gamma )$-order}
\author[B. Bela\"{\i}di ]{Benharrat Bela\"{\i}di}
\address{B. Bela\"{\i}di: Department of Mathematics, Laboratory of Pure and
Applied Mathematics, University of Mostaganem (UMAB), B. P. 227
Mostaganem-(Algeria).}
\email{benharrat.belaidi@univ-mosta.dz}
\thanks{}

\begin{abstract}
In this paper, we deal with the growth and oscillation of solutions of
higher order linear differential equations. Under the conditions that there
exists a coefficient which dominates the other coefficients by its lower $%
(\alpha ,\beta ,\gamma )$-order and lower $(\alpha ,\beta ,\gamma )$-type,
we obtain some growth and oscillation properties of solutions of such
equations which improve and extend some recently results of the author and
Biswas \cite{b8}.
\end{abstract}

\maketitle

\section{\textbf{Introduction}}

\qquad Throughout this paper, we assume that the reader is familiar with the
fundamental results and the standard notations of the Nevanlinna value
distribution theory of meromorphic functions \cite{go,1, 3}.

\qquad Nevanlinna theory has appeared to be a powerful tool in the field of
complex differential equations. For an introduction to the theory of
differential equations in the complex plane by using the Nevanlinna theory
see \cite{19}$.$ Active research in this field was started by Wittich \cite%
{w1,w2} and his students in the 1950's and 1960's. After their many authors
have investigated the complex differential equations%
\begin{equation}
f^{(k)}(z)+A_{k-1}(z)f^{(k-1)}(z)+\cdots +A_{1}(z)f^{\prime
}(z)+A_{0}(z)f(z)=0,  \label{1.1}
\end{equation}%
\begin{equation}
f^{(k)}(z)+A_{k-1}(z)f^{(k-1)}(z)+\cdots +A_{1}(z)f^{\prime
}(z)+A_{0}(z)f(z)=F  \label{1.2}
\end{equation}%
and achieved many valuable results when the coefficients $%
A_{0}(z),...,A_{k-1}(z),$ $(k\geq 2)$ and $F(z)$ in (\ref{1.1}) and (\ref%
{1.2}) are entire functions of finite order or finite iterated $p$-order or $%
(p,q)$-th order or $(p,q)$-$\varphi $ order; see (\cite{b1}, \cite{b2}, \cite%
{b3}, \cite{b}, \cite{bou}, \cite{c1}, \cite{gg}, \cite{20}, \cite{13}, \cite%
{19}, \cite{10}, \cite{11}, \cite{l}, \cite{STX}, \cite{t2}, \cite{t3}, \cite%
{15}).

\qquad Chyzhykov and Semochko \cite{CS} showed that both definitions of
iterated $p$-order (\cite{gl}, \cite{13}, \cite{DS}, \cite{SC}) and the $%
(p,q)$-th order (\cite{8}, \cite{9}) have the disadvantage that they do not
cover arbitrary growth (see \cite[Example 1.4]{CS}). They used more general
scale, called the $\varphi $-order (see \cite{CS}, \cite{S}) and the concept
of $\varphi $-order is used to study the growth of solutions of complex
differential equations in the whole complex plane and in the unit disc which
extend and improve many previous results (see \cite{b4, b5, CS,k,k1,S}).
Extending this notion, Long et al. \cite{L} recently introduce the concepts
of $[p,q]_{,\varphi }$-order and $[p,q]_{,\varphi }$-type (see \cite{L}) and
obtain some interesting results which considerably extend and improve some
earlier results. For details\ one may see \cite{L}.

\qquad The concept of generalized order $(\alpha ,\beta )$ of an entire
function was introduced by Sheremeta \cite{MNS}. Several authors made close
investigations on the properties of entire functions related to generalized
order $(\alpha ,\beta )$ in some different direction \cite{tb1,tb2}. On the
other hand, Mulyava et al. \cite{MST} have used the concept of $(\alpha
,\beta )$-order of an entire function in order to investigate the properties
of solutions of a heterogeneous differential equation of the second order
and obtained several remarkable results. For details about $(\alpha ,\beta )$%
-order\ one may see \cite{MST, MNS}.

\qquad Now, let $L$ be a class of continuous non-negative on $(-\infty
,+\infty )$ function $\alpha $ such that $\alpha (x)=\alpha (x_{0})\geq 0$
for $x\leq x_{0}$ and $\alpha (x)\uparrow +\infty $ as $x_{0}\leq
x\rightarrow +\infty $. We say that $\alpha \in L_{1}$, if $\alpha \in L$
and $\alpha (a+b)\leq \alpha (a)+\alpha (b)+c$ for all $a,b\geq R_{0}$ and
fixed $c\in (0,+\infty )$. Further, we say that $\alpha \in L_{2}$, if $%
\alpha \in L$ and $\alpha (x+O(1))=(1+o(1))\alpha (x)$ as $x\rightarrow
+\infty $. Finally, $\alpha \in L_{3}$, if $\alpha \in L$ and $\alpha
(a+b)\leq \alpha (a)+\alpha (b)$ for all $a,b\geq R_{0},$ i.e.$,$ $\alpha $
is subadditive. Clearly $L_{3}\subset L_{1}$.

\qquad Particularly, when $\alpha \in L_{3}$, then one can easily verify
that $\alpha (mr)\leq m\alpha (r),$ $m\geq 2$ is an integer. Up to a
normalization, subadditivity is implied by concavity. Indeed, if $\alpha (r)$
is concave on $[0,+\infty )$ and satisfies $\alpha (0)\geq 0$, then for $%
t\in \lbrack 0,1]$,%
\begin{equation*}
\alpha (tx)=\alpha (tx+(1-t)\cdot 0)\geq t\alpha (x)+(1-t)\alpha (0)\geq
t\alpha (x)\text{,}
\end{equation*}%
so that by choosing $t=\frac{a}{a+b}$ or $t=\frac{b}{a+b}$,%
\begin{eqnarray*}
\alpha (a+b) &=&\frac{a}{a+b}\alpha (a+b)+\frac{b}{a+b}\alpha (a+b) \\
&\leq &\alpha \left( \frac{a}{a+b}(a+b)\right) +\alpha \left( \frac{b}{a+b}%
(a+b)\right) \\
&=&\alpha (a)+\alpha (b)\text{, \ \ }a,b\geq 0\text{.}
\end{eqnarray*}%
As a non-decreasing, subadditive and unbounded function, $\alpha (r)$
satisfies%
\begin{equation*}
\alpha (r)\leq \alpha (r+R_{0})\leq \alpha (r)+\alpha (R_{0})
\end{equation*}%
for any $R_{0}\geq 0$. This yields that $\alpha (r)\sim \alpha (r+R_{0})$ as
$r\rightarrow +\infty $.

\qquad Let $\alpha ,$ $\beta $ and $\gamma $ satisfy the following two
conditions : (i) Always $\alpha \in L_{1},$ $\beta \in L_{2}$ and $\gamma
\in L_{3}$; and (ii) $\alpha (\log ^{[p]}x)=o(\beta (\log \gamma (x))),$ $%
p\geq 2,$ $\alpha (\log x)=o(\alpha \left( x\right) )$ and $\alpha
^{-1}(kx)=o\left( \alpha ^{-1}(x)\right) $ $\left( 0<k<1\right) $ as $%
x\rightarrow +\infty $.

\qquad Throughout this paper, we assume that $\alpha ,$ $\beta $ and $\gamma
$ always satisfy the above two conditions unless otherwise specifically
stated.

\qquad Recently, Heittokangas et al. \cite{HWWY} have introduced a new
concept of $\varphi $-order of entire and meromorphic functions considering $%
\varphi $ as subadditive function. For details\ one may see \cite{HWWY}.
Extending this notion, recently the author and Biswas \cite{b6} introduce
the definition of the $(\alpha ,\beta ,\gamma )$-order of a meromorphic
function.

\qquad The main aim of this paper is to study the growth and oscillation of
solutions of higher order linear differential equations using the concepts
of lower $(\alpha ,\beta ,\gamma )$-order and lower $(\alpha ,\beta ,\gamma
) $-type. In fact, some works relating to study the growth of solutions of
higher order linear differential equations using the concepts of $(\alpha
,\beta ,\gamma )$-order have been explored in \cite{b6}, \cite{b7} and \cite%
{b8}. In this paper, we obtain some results which improve and generalize
some previous results of the author and Biswas \cite{b8}.

\qquad For $x\in \lbrack 0,+\infty )$ and $k\in \mathbb{N}$ where $%
\mathbb{N}
$ is the set of all positive integers, define{\ iterations of the
exponential and logarithmic functions as }$\exp ^{[k]}x=\exp (\exp
^{[k-1]}x) $ and $\log ^{[k]}x=\log (\log ^{[k-1]}x)$ {with convention that $%
\log ^{[0]}x=x$, $\log ^{[-1]}x=\exp {x}$, $\exp ^{[0]}x=x$ and $\exp
^{[-1]}x=\log x$.}

\begin{definition}
\label{d1.1}(\cite{b6}) The $(\alpha ,\beta ,\gamma )$-order denoted by $%
\rho _{(\alpha ,\beta ,\gamma )}[f]$ of a meromorphic function $f$ is
defined by%
\begin{equation*}
\rho _{(\alpha ,\beta ,\gamma )}[f]=\underset{r\rightarrow +\infty }{\lim
\sup }\frac{\alpha \left( \log T\left( r,f\right) \right) }{\beta \left(
\log \gamma \left( r\right) \right) }\text{,}
\end{equation*}%
and for an entire function $f$, we define
\begin{equation*}
\rho _{(\alpha ,\beta ,\gamma )}[f]=\underset{r\rightarrow +\infty }{\lim
\sup }\frac{\alpha \left( \log T\left( r,f\right) \right) }{\beta \left(
\log \gamma \left( r\right) \right) }=\underset{r\rightarrow +\infty }{\lim
\sup }\frac{\alpha (\log ^{[2]}M(r,f))}{\beta \left( \log \gamma \left(
r\right) \right) }.
\end{equation*}
\end{definition}

\qquad Similar to Definition \ref{d1.1}, one can also define the lower $%
(\alpha ,\beta ,\gamma )$-order of a meromorphic function $f$ in the
following way:

\begin{definition}
\label{d1.2} The lower $(\alpha ,\beta ,\gamma )$-order denoted by $\mu
_{(\alpha ,\beta ,\gamma )}[f]$ of a meromorphic function $f$ is defined by%
\begin{equation*}
\mu _{(\alpha ,\beta ,\gamma )}[f]=\underset{r\rightarrow +\infty }{\lim
\inf }\frac{\alpha \left( \log T\left( r,f\right) \right) }{\beta \left(
\log \gamma \left( r\right) \right) }\text{,}
\end{equation*}%
for an entire function $f$, one can easily by Theorem 7.1 in \cite{go}
verify that
\begin{equation*}
\mu _{(\alpha ,\beta ,\gamma )}[f]=\underset{r\rightarrow +\infty }{\lim
\inf }\frac{\alpha \left( \log T\left( r,f\right) \right) }{\beta \left(
\log \gamma \left( r\right) \right) }=\underset{r\rightarrow +\infty }{\lim
\inf }\frac{\alpha (\log ^{[2]}M(r,f))}{\beta \left( \log \gamma \left(
r\right) \right) }\text{.}
\end{equation*}
\end{definition}

\begin{proposition}
\label{p1.1}(\cite{b6}) If $f$ is an entire function, then%
\begin{equation*}
\rho _{(\alpha (\log ),\beta ,\gamma )}[f]=\underset{r\rightarrow +\infty }{%
\lim \sup }\frac{\alpha (\log ^{[2]}T(r,f))}{\beta \left( \log \gamma \left(
r\right) \right) }=\underset{r\rightarrow +\infty }{\lim \sup }\frac{\alpha
(\log ^{[3]}M(r,f))}{\beta \left( \log \gamma \left( r\right) \right) }\text{%
,}
\end{equation*}%
and also by Theorem 7.1 in \cite{go}, one can easily verify that%
\begin{equation*}
\mu _{(\alpha (\log ),\beta ,\gamma )}[f]=\underset{r\rightarrow +\infty }{%
\lim \inf }\frac{\alpha (\log ^{[2]}T(r,f))}{\beta \left( \log \gamma \left(
r\right) \right) }=\underset{r\rightarrow +\infty }{\lim \inf }\frac{\alpha
(\log ^{[3]}M(r,f))}{\beta \left( \log \gamma \left( r\right) \right) }\text{%
,}
\end{equation*}%
where $(\alpha (\log ),\beta ,\gamma )$-order denoted by $\rho _{(\alpha
(\log ),\beta ,\gamma )}[f]$ and lower $(\alpha (\log ),\beta ,\gamma )$%
-order denoted by $\mu _{(\alpha (\log ),\beta ,\gamma )}[f]$.
\end{proposition}

\qquad Now to compare the relative growth of two meromorphic functions
having same non zero finite $(\alpha ,\beta ,\gamma )$-order or non zero
finite lower $(\alpha ,\beta ,\gamma )$-order, one may introduce the
definitions of $(\alpha ,\beta ,\gamma )$-type and lower $(\alpha ,\beta
,\gamma )$-type in the following manner:

\begin{definition}
\label{d1.3}(\cite{b8}) The $(\alpha ,\beta ,\gamma )$-type denoted by $\tau
_{(\alpha ,\beta ,\gamma )}[f]$ of a meromorphic function $f$ with $0<\rho
_{(\alpha ,\beta ,\gamma )}[f]<+\infty $ is defined by%
\begin{equation*}
\tau _{(\alpha ,\beta ,\gamma )}[f]=\underset{r\rightarrow +\infty }{\lim
\sup }\frac{\exp (\alpha \left( \log T\left( r,f\right) \right) )}{\left(
\exp \left( \beta \left( \log \gamma \left( r\right) \right) \right) \right)
^{\rho _{(\alpha ,\beta ,\gamma )}[f]}}\text{.}
\end{equation*}%
If $f$ is an entire function with $\rho _{(\alpha ,\beta ,\gamma )}[f]\in
(0,+\infty )$, then the $(\alpha ,\beta ,\gamma )$-type of $f$ is defined by%
\begin{equation*}
\tau _{(\alpha ,\beta ,\gamma ),M}[f]=\underset{r\rightarrow +\infty }{\lim
\sup }\frac{\exp (\alpha (\log ^{[2]}M(r,f)))}{\left( \exp \left( \beta
\left( \log \gamma \left( r\right) \right) \right) \right) ^{\rho _{(\alpha
,\beta ,\gamma )}[f]}}\text{.}
\end{equation*}
\end{definition}

\begin{definition}
\label{d1.4} The lower $(\alpha ,\beta ,\gamma )$-type denoted by $%
\underline{\tau }_{(\alpha ,\beta ,\gamma )}[f]$ of a meromorphic function $%
f $ with $0<\mu _{(\alpha ,\beta ,\gamma )}[f]<+\infty $ is defined by%
\begin{equation*}
\underline{\tau }_{(\alpha ,\beta ,\gamma )}[f]=\underset{r\rightarrow
+\infty }{\lim \inf }\frac{\exp (\alpha \left( \log T\left( r,f\right)
\right) )}{\left( \exp \left( \beta \left( \log \gamma \left( r\right)
\right) \right) \right) ^{\mu _{(\alpha ,\beta ,\gamma )}[f]}}\text{.}
\end{equation*}%
If $f$ is an entire function with $\mu _{(\alpha ,\beta ,\gamma )}[f]\in
(0,+\infty )$, then the lower $(\alpha ,\beta ,\gamma )$-type of $f$ is
defined by%
\begin{equation*}
\underline{\tau }_{(\alpha ,\beta ,\gamma ),M}[f]=\underset{r\rightarrow
+\infty }{\lim \inf }\frac{\exp (\alpha (\log ^{[2]}M(r,f)))}{\left( \exp
\left( \beta \left( \log \gamma \left( r\right) \right) \right) \right)
^{\mu _{(\alpha ,\beta ,\gamma )}[f]}}\text{.}
\end{equation*}
\end{definition}

\qquad In order to study the oscillation properties of solutions of (\ref%
{1.1}) and (\ref{1.2}), we define the $(\alpha ,\beta ,\gamma )$-exponent
convergence of the zero-sequence of a meromorphic function $f$ in the
following way:

\begin{definition}
\label{d1.5}(\cite{b6}) The $(\alpha ,\beta ,\gamma )$-exponent convergence
of the zero-sequence denoted by $\lambda _{(\alpha ,\beta ,\gamma )}[f]$ of
a meromorphic function $f$ is defined by%
\begin{equation*}
\lambda _{(\alpha ,\beta ,\gamma )}[f]=\underset{r\rightarrow +\infty }{\lim
\sup }\frac{\alpha (\log n(r,1/f))}{\beta (\log \gamma (r))}=\underset{%
r\rightarrow +\infty }{\lim \sup }\frac{\alpha (\log N(r,1/f))}{\beta (\log
\gamma (r))}\text{.}
\end{equation*}%
Analogously, the $(\alpha ,\beta ,\gamma )$-exponent convergence of the
distinct zero-sequence denoted by $\overline{\lambda }_{(\alpha ,\beta
,\gamma )}[f]$ of $f$ is defined by%
\begin{equation*}
\overline{\lambda }_{(\alpha ,\beta ,\gamma )}[f]=\underset{r\rightarrow
+\infty }{\lim \sup }\frac{\alpha (\log \overline{n}(r,1/f))}{\beta (\log
\gamma (r))}=\underset{r\rightarrow +\infty }{\lim \sup }\frac{\alpha (\log
\overline{N}(r,1/f))}{\beta (\log \gamma (r))}\text{.}
\end{equation*}
\end{definition}

\qquad Accordingly, the values%
\begin{equation*}
\lambda _{(\alpha (\log ),\beta ,\gamma )}[f]=\underset{r\rightarrow +\infty
}{\lim \sup }\frac{\alpha (\log ^{[2]}n(r,1/f))}{\beta (\log \gamma (r))}=%
\underset{r\rightarrow +\infty }{\lim \sup }\frac{\alpha (\log
^{[2]}N(r,1/f))}{\beta (\log \gamma (r))}\text{ }
\end{equation*}%
and%
\begin{equation*}
\overline{\lambda }_{(\alpha (\log ),\beta ,\gamma )}[f]=\underset{%
r\rightarrow +\infty }{\lim \sup }\frac{\alpha (\log ^{[2]}\overline{n}%
(r,1/f))}{\beta (\log \gamma (r))}=\underset{r\rightarrow +\infty }{\lim
\sup }\frac{\alpha (\log ^{[2]}\overline{N}(r,1/f))}{\beta (\log \gamma (r))}
\end{equation*}%
are respectively called as $(\alpha (\log ),\beta ,\gamma )$-exponent
convergence of the zero-sequence and $(\alpha (\log ),\beta ,\gamma )$%
-exponent convergence of the distinct zero-sequence of a meromorphic
function $f$.\newline
\qquad Similar to Definition \ref{d1.5}, one can also define the lower $%
(\alpha ,\beta ,\gamma )$-exponent convergence of the zero-sequence of a
meromorphic function $f$ in the following way:

\begin{definition}
\label{d1.6} The lower $(\alpha ,\beta ,\gamma )$-exponent convergence of
the zero-sequence denoted by $\underline{\lambda }_{(\alpha ,\beta ,\gamma
)}[f]$ of a meromorphic function $f$ is defined by%
\begin{equation*}
\underline{\lambda }_{(\alpha ,\beta ,\gamma )}[f]=\underset{r\rightarrow
+\infty }{\lim \inf }\frac{\alpha (\log n(r,1/f))}{\beta (\log \gamma (r))}=%
\underset{r\rightarrow +\infty }{\lim \inf }\frac{\alpha (\log N(r,1/f))}{%
\beta (\log \gamma (r))}\text{.}
\end{equation*}%
Analogously, the lower $(\alpha ,\beta ,\gamma )$-exponent convergence of
the distinct zero-sequence denoted by $\underline{\overline{\lambda }}%
_{(\alpha ,\beta ,\gamma )}[f]$ of $f$ is defined by%
\begin{equation*}
\underline{\overline{\lambda }}_{(\alpha ,\beta ,\gamma )}[f]=\underset{%
r\rightarrow +\infty }{\lim \inf }\frac{\alpha (\log \overline{n}(r,1/f))}{%
\beta (\log \gamma (r))}=\underset{r\rightarrow +\infty }{\lim \inf }\frac{%
\alpha (\log \overline{N}(r,1/f))}{\beta (\log \gamma (r))}\text{.}
\end{equation*}
\end{definition}

\qquad Accordingly, the values%
\begin{equation*}
\underline{\lambda }_{(\alpha (\log ),\beta ,\gamma )}[f]=\underset{%
r\rightarrow +\infty }{\lim \inf }\frac{\alpha (\log ^{[2]}n(r,1/f))}{\beta
(\log \gamma (r))}=\underset{r\rightarrow +\infty }{\lim \inf }\frac{\alpha
(\log ^{[2]}N(r,1/f))}{\beta (\log \gamma (r))}\text{ }
\end{equation*}%
and%
\begin{equation*}
\underline{\overline{\lambda }}_{(\alpha (\log ),\beta ,\gamma )}[f]=%
\underset{r\rightarrow +\infty }{\lim \inf }\frac{\alpha (\log ^{[2]}%
\overline{n}(r,1/f))}{\beta (\log \gamma (r))}=\underset{r\rightarrow
+\infty }{\lim \inf }\frac{\alpha (\log ^{[2]}\overline{N}(r,1/f))}{\beta
(\log \gamma (r))}
\end{equation*}%
are respectively called as lower $(\alpha (\log ),\beta ,\gamma )$-exponent
convergence of the zero-sequence and lower $(\alpha (\log ),\beta ,\gamma )$%
-exponent convergence of the distinct zero-sequence of a meromorphic
function $f$.

\begin{proposition}
\label{p1.2}(\cite{b6}) Let $f_{1}(z),$ $f_{2}(z)$ be nonconstant
meromorphic functions with $\rho _{(\alpha (\log ),\beta ,\gamma )}[f_{1}]$
and $\rho _{(\alpha (\log ),\beta ,\gamma )}[f_{2}]$ as their $(\alpha
\left( \log \right) ,\beta ,\gamma )$-order. Then\newline
(i) $\rho _{(\alpha (\log ),\beta ,\gamma )}[f_{1}\pm f_{2}]\leq \max \{\rho
_{(\alpha (\log ),\beta ,\gamma )}[f_{1}],$ $\rho _{(\alpha (\log ),\beta
,\gamma )}[f_{2}]\}$;\newline
(ii) $\rho _{(\alpha (\log ),\beta ,\gamma )}[f_{2}\cdot f_{2}]\leq \max
\{\rho _{(\alpha (\log ),\beta ,\gamma )}[f_{1}],$ $\rho _{(\alpha (\log
),\beta ,\gamma )}[f_{2}]\}$;\newline
(iii) If $\rho _{(\alpha (\log ),\beta ,\gamma )}[f_{1}]\neq \rho _{(\alpha
(\log ),\beta ,\gamma )}[f_{2}]$, then
\begin{equation*}
\rho _{(\alpha (\log ),\beta ,\gamma )}[f_{1}\pm f_{2}]=\max \{\rho
_{(\alpha (\log ),\beta ,\gamma )}[f_{1}],\rho _{(\alpha (\log ),\beta
,\gamma )}[f_{2}]\};
\end{equation*}%
(iv) If $\rho _{(\alpha (\log ),\beta ,\gamma )}[f_{1}]\neq \rho _{(\alpha
(\log ),\beta ,\gamma )}[f_{2}]$, then
\begin{equation*}
\rho _{(\alpha (\log ),\beta ,\gamma )}[f_{2}\cdot f_{2}]=\max \{\rho
_{(\alpha (\log ),\beta ,\gamma )}[f_{1}],\rho _{(\alpha (\log ),\beta
,\gamma )}[f_{2}]\}.
\end{equation*}
\end{proposition}

\qquad By using the properties $T(r,f)=T(r,\frac{1}{f})+O(1)$ and $%
T(r,af)=T(r,f)+O(1)$, $a\in
\mathbb{C}
\setminus \{0\}$, one can obtain the following result.

\begin{proposition}
\label{p1.3} (\cite{b8}) Let $f$ be a nonconstant meromorphic function. Then%
\newline
(i) $\rho _{(\alpha ,\beta ,\gamma )}[\frac{1}{f}]=\rho _{(\alpha ,\beta
,\gamma )}[f]$ $\left( f\not\equiv 0\right) ;$\newline
(ii) $\rho _{(\alpha (\log ),\beta ,\gamma )}[\frac{1}{f}]=\rho _{(\alpha
(\log ),\beta ,\gamma )}[f]$ $\left( f\not\equiv 0\right) ;$\newline
(iii) If $a\in
\mathbb{C}
\setminus \{0\}$, then $\rho _{(\alpha ,\beta ,\gamma )}[af]=\rho _{(\alpha
,\beta ,\gamma )}[f]$ and $\tau _{(\alpha ,\beta ,\gamma )}[af]=\tau
_{(\alpha ,\beta ,\gamma )}[f]$ if $0<\rho _{(\alpha ,\beta ,\gamma
)}[f]<+\infty ;$\newline
(iii) If $a\in
\mathbb{C}
\setminus \{0\}$, then $\rho _{(\alpha (\log ),\beta ,\gamma )}[af]=\rho
_{(\alpha (\log ),\beta ,\gamma )}[f]$ and $\tau _{(\alpha (\log ),\beta
,\gamma )}[af]=\tau _{(\alpha (\log ),\beta ,\gamma )}[f]$ if $0<\rho
_{(\alpha \left( \log \right) ,\beta ,\gamma )}[f]<+\infty $.
\end{proposition}

\begin{proposition}
\label{p1.4} Let $f,$ $g$ be nonconstant meromorphic functions with $\rho
_{(\alpha \left( \log \right) ,\beta ,\gamma )}[f]$ as $(\alpha \left( \log
\right) ,\beta ,\gamma )$-order and $\mu _{(\alpha \left( \log \right)
,\beta ,\gamma )}[g]$ as lower $(\alpha \left( \log \right) ,\beta ,\gamma )$%
-order. Then%
\begin{equation*}
\mu _{(\alpha \left( \log \right) ,\beta ,\gamma )}\left( f+g\right) \leq
\max \left\{ \rho _{(\alpha \left( \log \right) ,\beta ,\gamma )}\left(
f\right) ,\mu _{(\alpha \left( \log \right) ,\beta ,\gamma )}\left( g\right)
\right\}
\end{equation*}%
\textit{and}%
\begin{equation*}
\mu _{(\alpha \left( \log \right) ,\beta ,\gamma )}\left( fg\right) \leq
\max \left\{ \rho _{(\alpha \left( \log \right) ,\beta ,\gamma )}\left(
f\right) ,\mu _{(\alpha \left( \log \right) ,\beta ,\gamma )}\left( g\right)
\right\} .
\end{equation*}%
\textit{Furthermore}, \textit{if }$\mu _{(\alpha \left( \log \right) ,\beta
,\gamma )}\left( g\right) >\rho _{(\alpha \left( \log \right) ,\beta ,\gamma
)}\left( f\right) ,$ \textit{then we obtain}%
\begin{equation*}
\mu _{(\alpha \left( \log \right) ,\beta ,\gamma )}\left( f+g\right) =\mu
_{(\alpha \left( \log \right) ,\beta ,\gamma )}\left( fg\right) =\mu
_{(\alpha \left( \log \right) ,\beta ,\gamma )}\left( g\right) .
\end{equation*}
\end{proposition}

\begin{proof}
Without loss of generality, we assume that $\rho _{(\alpha \left( \log
\right) ,\beta ,\gamma )}\left( f\right) <+\infty $ and $\mu _{(\alpha
\left( \log \right) ,\beta ,\gamma )}\left( g\right) <+\infty .$ From the
definition of the lower $(\alpha \left( \log \right) ,\beta ,\gamma )$%
-order, there exists a sequence $r_{n}\longrightarrow +\infty $ $\left(
n\longrightarrow +\infty \right) $ such that
\begin{equation*}
\underset{n\longrightarrow +\infty }{\lim }\frac{\alpha \left( \log
^{[2]}T\left( r_{n},g\right) \right) }{\beta \left( \log \gamma \left(
r_{n}\right) \right) }=\mu _{(\alpha \left( \log \right) ,\beta ,\gamma
)}\left( g\right) .
\end{equation*}%
Then, for any given $\varepsilon >0,$ there exists a positive integer $N_{1}$
such that%
\begin{equation*}
T(r_{n},g)\leq \exp ^{[2]}\left\{ \alpha ^{-1}\left( \left( \mu _{(\alpha
\left( \log \right) ,\beta ,\gamma )}\left( g\right) +\varepsilon \right)
\beta \left( \log \gamma \left( r_{n}\right) \right) \right) \right\}
\end{equation*}%
holds for $n>N_{1}.$ From the definition of the $(\alpha \left( \log \right)
,\beta ,\gamma )-$order, for any given $\varepsilon >0,$ there exists a
positive number $R$ such that
\begin{equation*}
T(r,f)\leq \exp ^{[2]}\left\{ \alpha ^{-1}\left( \left( \rho _{(\alpha
\left( \log \right) ,\beta ,\gamma )}\left( f\right) +\varepsilon \right)
\beta \left( \log \gamma \left( r\right) \right) \right) \right\}
\end{equation*}%
holds for $r\geq R.$ Since $r_{n}\longrightarrow +\infty $ $\left(
n\longrightarrow +\infty \right) ,$ there exists a positive integer $N_{2}$
such that $r_{n}>R,$ and thus
\begin{equation*}
T(r_{n},f)\leq \exp ^{[2]}\left\{ \alpha ^{-1}\left( \left( \rho _{(\alpha
\left( \log \right) ,\beta ,\gamma )}\left( f\right) +\varepsilon \right)
\beta \left( \log \gamma \left( r_{n}\right) \right) \right) \right\}
\end{equation*}%
holds for $n>N_{2}.$ Note that%
\begin{equation*}
T\left( r,f+g\right) \leq T\left( r,f\right) +T\left( r,g\right) +\ln 2
\end{equation*}%
and
\begin{equation*}
T\left( r,fg\right) \leq T\left( r,f\right) +T\left( r,g\right) .
\end{equation*}%
Then, for any given $\varepsilon >0,$ we have for $n>\max \left\{
N_{1},N_{2}\right\} $%
\begin{equation*}
T\left( r_{n},f+g\right) \leq T\left( r_{n},f\right) +T\left( r_{n},g\right)
+\ln 2
\end{equation*}%
\begin{equation*}
\leq \exp ^{[2]}\left\{ \alpha ^{-1}\left( \left( \rho _{(\alpha \left( \log
\right) ,\beta ,\gamma )}\left( f\right) +\varepsilon \right) \beta \left(
\log \gamma \left( r_{n}\right) \right) \right) \right\}
\end{equation*}%
\begin{equation*}
+\exp ^{[2]}\left\{ \alpha ^{-1}\left( \left( \mu _{(\alpha \left( \log
\right) ,\beta ,\gamma )}\left( g\right) +\varepsilon \right) \beta \left(
\log \gamma \left( r_{n}\right) \right) \right) \right\} +\ln 2
\end{equation*}%
\begin{equation}
\leq 3\exp ^{[2]}\left\{ \alpha ^{-1}\left( \left( \max \left\{ \rho
_{(\alpha \left( \log \right) ,\beta ,\gamma )}\left( f\right) ,\mu
_{(\alpha \left( \log \right) ,\beta ,\gamma )}\left( g\right) \right\}
+\varepsilon \right) \beta \left( \log \gamma \left( r_{n}\right) \right)
\right) \right\}  \label{1.3}
\end{equation}%
and%
\begin{equation*}
T\left( r_{n},fg\right) \leq T\left( r_{n},f\right) +T\left( r_{n},g\right)
\end{equation*}%
\begin{equation}
\leq 2\exp ^{[2]}\left\{ \alpha ^{-1}\left( \left( \max \left\{ \rho
_{(\alpha \left( \log \right) ,\beta ,\gamma )}\left( f\right) ,\mu
_{(\alpha \left( \log \right) ,\beta ,\gamma )}\left( g\right) \right\}
+\varepsilon \right) \beta \left( \log \gamma \left( r_{n}\right) \right)
\right) \right\} .  \label{1.4}
\end{equation}%
Since $\varepsilon >0$ is arbitrary, then from (\ref{1.3}) and (\ref{1.4})$,$
we easily obtain%
\begin{equation}
\mu _{(\alpha \left( \log \right) ,\beta ,\gamma )}\left( f+g\right) \leq
\max \left\{ \rho _{(\alpha \left( \log \right) ,\beta ,\gamma )}\left(
f\right) ,\mu _{(\alpha \left( \log \right) ,\beta ,\gamma )}\left( g\right)
\right\}  \label{1.5}
\end{equation}%
and%
\begin{equation}
\mu _{(\alpha \left( \log \right) ,\beta ,\gamma )}\left( fg\right) \leq
\max \left\{ \rho _{(\alpha \left( \log \right) ,\beta ,\gamma )}\left(
f\right) ,\mu _{(\alpha \left( \log \right) ,\beta ,\gamma )}\left( g\right)
\right\} .  \label{1.6}
\end{equation}%
Suppose now that $\mu _{(\alpha \left( \log \right) ,\beta ,\gamma )}\left(
g\right) >\rho _{(\alpha \left( \log \right) ,\beta ,\gamma )}\left(
f\right) .$ Considering that%
\begin{equation}
T\left( r,g\right) =T\left( r,f+g-f\right) \leq T\left( r,f+g\right)
+T\left( r,f\right) +\ln 2  \label{1.7}
\end{equation}%
and%
\begin{equation*}
T\left( r,g\right) =T\left( r,\frac{fg}{f}\right) \leq T\left( r,fg\right)
+T\left( r,\frac{1}{f}\right)
\end{equation*}%
\begin{equation}
=T\left( r,fg\right) +T\left( r,f\right) +O\left( 1\right) .  \label{1.8}
\end{equation}%
By (\ref{1.7}), (\ref{1.8}) and the same method as above we obtain that%
\begin{equation*}
\mu _{(\alpha \left( \log \right) ,\beta ,\gamma )}\left( g\right) \leq \max
\left\{ \mu _{(\alpha \left( \log \right) ,\beta ,\gamma )}\left( f+g\right)
,\rho _{(\alpha \left( \log \right) ,\beta ,\gamma )}\left( f\right) \right\}
\end{equation*}%
\begin{equation}
=\mu _{(\alpha \left( \log \right) ,\beta ,\gamma )}\left( f+g\right)
\label{1.9}
\end{equation}%
and%
\begin{equation}
\mu _{(\alpha \left( \log \right) ,\beta ,\gamma )}\left( g\right) \leq \max
\left\{ \mu _{(\alpha \left( \log \right) ,\beta ,\gamma )}\left( fg\right)
,\rho _{(\alpha \left( \log \right) ,\beta ,\gamma )}\left( f\right)
\right\} =\mu _{(\alpha \left( \log \right) ,\beta ,\gamma )}\left(
fg\right) .  \label{1.10}
\end{equation}%
By using (\ref{1.5}) and (\ref{1.9}) we obtain $\mu _{(\alpha \left( \log
\right) ,\beta ,\gamma )}\left( f+g\right) =\mu _{(\alpha \left( \log
\right) ,\beta ,\gamma )}\left( g\right) $ and by (\ref{1.6}) and (\ref{1.10}%
)$,$ we get $\mu _{(\alpha \left( \log \right) ,\beta ,\gamma )}\left(
fg\right) =\mu _{(\alpha \left( \log \right) ,\beta ,\gamma )}\left(
g\right) .$
\end{proof}

\section{\textbf{Main Results}}

\qquad Very recently the author and Biswas have investigated the growth of
solutions of\ equation $\left( \ref{1.1}\right) $ and established the
following two results.

\begin{theorem}
\label{t2.1} (\cite{b8}) Let $A_{0}(z),$ $A_{1}(z),...,A_{k-1}(z)$ be entire
functions such that $\rho _{(\alpha ,\beta ,\gamma )}[A_{0}]>\max \{\rho
_{(\alpha ,\beta ,\gamma )}[A_{j}],$ $j=1,...,k-1\}$. Then every solution $%
f(z)\not\equiv 0$ of $\left( \ref{1.1}\right) $ satisfies $\rho _{(\alpha
(\log ),\beta ,\gamma )}[f]=\rho _{(\alpha ,\beta ,\gamma )}[A_{0}]$.
\end{theorem}

\begin{theorem}
\label{t2.2} (\cite{b8}) Let $A_{0}(z),$ $A_{1}(z),...,A_{k-1}(z)$ be entire
functions. Assume that%
\begin{equation*}
\max \{\rho _{(\alpha ,\beta ,\gamma )}[A_{j}],j=1,...,k-1\}\leq \rho
_{(\alpha ,\beta ,\gamma )}[A_{0}]=\rho _{0}<+\infty
\end{equation*}%
and%
\begin{equation*}
\max \{\tau _{(\alpha ,\beta ,\gamma ),M}[A_{j}]:\rho _{(\alpha ,\beta
,\gamma )}[A_{j}]=\rho _{(\alpha ,\beta ,\gamma )}[A_{0}]>0\}<\tau _{(\alpha
,\beta ,\gamma ),M}[A_{0}]=\tau _{M}\text{.}
\end{equation*}%
Then every solution $f(z)\not\equiv 0$ of $\left( \ref{1.1}\right) $
satisfies $\rho _{(\alpha (\log ),\beta ,\gamma )}[f]=\rho _{(\alpha ,\beta
,\gamma )}[A_{0}].$
\end{theorem}

\qquad Theorems \ref{t2.1} and \ref{t2.2} concerned the growth properties of
solutions of $\left( \ref{1.1}\right) $, when $A_{0}$ is dominating the
others coefficients by its $(\alpha ,\beta ,\gamma )$-order and $(\alpha
,\beta ,\gamma )$-type. Thus, the natural question which arises: If $A_{0}$
is dominating coefficient with its lower $(\alpha ,\beta ,\gamma )$-order
and lower $(\alpha ,\beta ,\gamma )$-type, what can we say about the growth
of solutions of $\left( \ref{1.1}\right) ?$ The following results give
answer to this question.

\begin{theorem}
\label{t2.3} \textit{Let }$A_{0}\left( z\right) ,...,$\textit{\ }$%
A_{k-1}\left( z\right) $\textit{\ be entire functions}$.$\textit{\ Assume
that }$\max \{\rho _{(\alpha ,\beta ,\gamma )}[A_{j}]:j=1,...,k-1\}<\mu
_{(\alpha ,\beta ,\gamma )}[A_{0}]\leq \rho _{(\alpha ,\beta ,\gamma
)}[A_{0}]<+\infty .$\textit{\ Then every solution }$f\not\equiv 0$\textit{\
of }$\left( \ref{1.1}\right) $\textit{\ satisfies}%
\begin{equation*}
\underline{\overline{\lambda }}_{(\alpha \left( \log \right) ,\beta ,\gamma
)}[f-g]=\mu _{(\alpha ,\beta ,\gamma )}[A_{0}]=\mu _{(\alpha (\log ),\beta
,\gamma )}[f]
\end{equation*}%
\textit{\ }
\begin{equation*}
\leq \rho _{(\alpha (\log ),\beta ,\gamma )}[f]=\rho _{(\alpha ,\beta
,\gamma )}[A_{0}]=\overline{\lambda }_{(\alpha (\log ),\beta ,\gamma )}[f-g],
\end{equation*}%
where $g\not\equiv 0$ is an entire function satisfying $\rho _{(\alpha (\log
),\beta ,\gamma )}[g]<\mu _{(\alpha ,\beta ,\gamma )}[A_{0}].$
\end{theorem}

\begin{theorem}
\label{t2.4} Let $A_{0}(z),$ $A_{1}(z),...,A_{k-1}(z)$ be entire functions.
Assume that\textit{\ }%
\begin{equation*}
\max \{\rho _{(\alpha ,\beta ,\gamma )}[A_{j}]:j=1,...,k-1\}\leq \mu
_{(\alpha ,\beta ,\gamma )}[A_{0}]
\end{equation*}%
\begin{equation*}
\leq \rho _{(\alpha ,\beta ,\gamma )}[A_{0}]=\rho <+\infty \text{ }\left(
0<\rho <+\infty \right)
\end{equation*}%
\textit{and}%
\begin{equation*}
\mathit{\ }\tau _{1}=\max \{\tau _{(\alpha ,\beta ,\gamma ),M}\left[ A_{j}%
\right] :\rho _{(\alpha ,\beta ,\gamma )}[A_{j}]=\mu _{(\alpha ,\beta
,\gamma )}[A_{0}]>0\}<\underline{\tau }_{(\alpha ,\beta ,\gamma ),M}[A_{0}]
\end{equation*}%
\begin{equation*}
=\tau \mathit{\ }\left( 0<\tau <+\infty \right) .
\end{equation*}%
\textit{Then every solution }$f\not\equiv 0$\textit{\ of }$\left( \ref{1.1}%
\right) $\textit{\ satisfies}%
\begin{equation*}
\underline{\overline{\lambda }}_{(\alpha \left( \log \right) ,\beta ,\gamma
)}[f-g]=\mu _{(\alpha ,\beta ,\gamma )}[A_{0}]=\mu _{(\alpha (\log ),\beta
,\gamma )}[f]
\end{equation*}%
\textit{\ }
\begin{equation*}
\leq \rho _{(\alpha (\log ),\beta ,\gamma )}[f]=\rho _{(\alpha ,\beta
,\gamma )}[A_{0}]=\overline{\lambda }_{(\alpha (\log ),\beta ,\gamma )}[f-g],
\end{equation*}%
where $g\not\equiv 0$ is an entire function satisfying $\rho _{(\alpha (\log
),\beta ,\gamma )}[g]<\mu _{(\alpha ,\beta ,\gamma )}[A_{0}].$
\end{theorem}

\begin{theorem}
\label{t2.5} \textit{Let }$A_{0}\left( z\right) ,...,$\textit{\ }$%
A_{k-1}\left( z\right) $\textit{\ be entire functions}$.$\textit{\ Assume
that }$\max \{\rho _{(\alpha ,\beta ,\gamma )}[A_{j}]:j=1,...,k-1\}\leq \mu
_{(\alpha ,\beta ,\gamma )}[A_{0}]$\textit{\ }$<+\infty $\textit{\ and }%
\begin{equation*}
\underset{r\rightarrow +\infty }{\lim \sup }\frac{\overset{k-1}{%
\sum\limits_{j=1}}m\left( r,A_{j}\right) }{m\left( r,A_{0}\right) }<1.
\end{equation*}%
\textit{Then every solution }$f\not\equiv 0$\textit{\ of }$\left( \ref{1.1}%
\right) $\textit{\ satisfies}%
\begin{equation*}
\underline{\overline{\lambda }}_{(\alpha \left( \log \right) ,\beta ,\gamma
)}[f-g]=\mu _{(\alpha ,\beta ,\gamma )}[A_{0}]=\mu _{(\alpha (\log ),\beta
,\gamma )}[f]
\end{equation*}%
\textit{\ }
\begin{equation*}
\leq \rho _{(\alpha (\log ),\beta ,\gamma )}[f]=\rho _{(\alpha ,\beta
,\gamma )}[A_{0}]=\overline{\lambda }_{(\alpha (\log ),\beta ,\gamma )}[f-g],
\end{equation*}%
where $g\not\equiv 0$ is an entire function satisfying $\rho _{(\alpha (\log
),\beta ,\gamma )}[g]<\mu _{(\alpha ,\beta ,\gamma )}[A_{0}].$
\end{theorem}

\begin{theorem}
\label{t2.6} \textit{Let }$A_{0}\left( z\right) ,...,$\textit{\ }$%
A_{k-1}\left( z\right) $\textit{\ be entire functions such that }$%
A_{0}\left( z\right) $ is transcendental$.$\textit{\ Assume that }$\max
\{\rho _{(\alpha ,\beta ,\gamma )}[A_{j}]:j=1,...,k-1\}\leq \mu _{(\alpha
,\beta ,\gamma )}[A_{0}]$\textit{\ }$=\rho _{(\alpha ,\beta ,\gamma
)}[A_{0}]<+\infty $\textit{\ and }%
\begin{equation*}
\underset{r\rightarrow +\infty }{\lim \inf }\frac{\overset{k-1}{%
\sum\limits_{j=1}}m\left( r,A_{j}\right) }{m\left( r,A_{0}\right) }<1,\text{
}r\notin E,
\end{equation*}%
where $E$ is a set of $r$ of finite linear measure. \textit{Then every
solution }$f\not\equiv 0$\textit{\ of }$\left( \ref{1.1}\right) $\textit{\
satisfies}%
\begin{equation*}
\underline{\overline{\lambda }}_{(\alpha \left( \log \right) ,\beta ,\gamma
)}[f-g]=\mu _{(\alpha ,\beta ,\gamma )}[A_{0}]=\mu _{(\alpha (\log ),\beta
,\gamma )}[f]
\end{equation*}%
\textit{\ }
\begin{equation*}
=\rho _{(\alpha (\log ),\beta ,\gamma )}[f]=\rho _{(\alpha ,\beta ,\gamma
)}[A_{0}]=\overline{\lambda }_{(\alpha (\log ),\beta ,\gamma )}[f-g],
\end{equation*}%
where $g\not\equiv 0$ is an entire function satisfying $\rho _{(\alpha (\log
),\beta ,\gamma )}[g]<\mu _{(\alpha ,\beta ,\gamma )}[A_{0}].$
\end{theorem}

\section{\textbf{Preliminary Lemmas}}

\qquad In this section we present some lemmas which will be needed in the
sequel. First, {we denote the} Lebesgue linear measure of a set $E\subset %
\left[ 0,+\infty \right) $ by $m\left( E\right) =\underset{F}{\int }dt,$ and
the logarithmic measure of a set $F\subset \left( 1,+\infty \right) $ by $%
m_{l}\left( F\right) =\underset{F}{\int }\frac{dt}{t}.$

\qquad The following result due to Gundersen \cite{gs} plays an important
role in the theory of complex differential equations.

\begin{lemma}
\label{l3.1}$($\cite{gs})\textit{\ Let }$f$\textit{\ be a transcendental
meromorphic function, and let }$\chi >1$\textit{\ be a given constant. Then
there exist a set }$E_{1}\subset \left( 1,\infty \right) $\textit{\ with
finite logarithmic measure and a constant }$B>0$\textit{\ that depends only
on }$\chi $\textit{\ and }$i,j$\textit{\ }$(0\leq i<j\leq k)$\textit{, such
that for all }$z$\textit{\ satisfying }$\left\vert z\right\vert =r\notin
\lbrack 0,1]\cup E_{1}$\textit{, we have}%
\begin{equation*}
\left\vert \frac{f^{\left( j\right) }(z)}{f^{\left( i\right) }(z)}%
\right\vert \leq B\left\{ \frac{T(\chi r,f)}{r}\left( \log ^{\chi }r\right)
\log T(\chi r,f)\right\} ^{j-i}.
\end{equation*}
\end{lemma}

\begin{lemma}
\label{l3.2} Let $f$ be a meromorphic function with $\mu _{(\alpha \left(
\log \right) ,\beta ,\gamma )}[f]=\mu <+\infty $. \textit{Then there exists
a set} $E_{2}\subset \left( 1,+\infty \right) $ \textit{with infinite
logarithmic measure such that for }$r\in E_{2}\subset \left( 1,+\infty
\right) ,$ \textit{we have for any given }$\varepsilon >0$\textit{\ }%
\begin{equation*}
T\left( r,f\right) <\exp ^{\left[ 2\right] }\left\{ \alpha ^{-1}\left(
\left( \mu +\varepsilon \right) \beta \left( \log \gamma \left( r\right)
\right) \right) \right\} .
\end{equation*}
\end{lemma}

\begin{proof}
\textbf{\ }The definition of lower $(\alpha \left( \log \right) ,\beta
,\gamma )$-order implies that there exists a sequence $\left\{ r_{n}\right\}
_{n=1}^{+\infty }$ tending to $\infty $ satisfying $\left( 1+\frac{1}{n}%
\right) r_{n}<r_{n+1}$ and
\begin{equation*}
\lim\limits_{r_{n}\rightarrow \infty }\frac{\alpha (\log ^{\left[ 2\right]
}T(r_{n},f))}{\beta (\log \gamma (r_{n}))}=\mu _{(\alpha ,\beta ,\gamma
)}[f].
\end{equation*}%
Then for any given $\varepsilon >0$, there exists an integer $n_{1}$ such
that for all $n\geq n_{1}$,
\begin{equation*}
T\left( r_{n},f\right) <\exp ^{\left[ 2\right] }\left\{ \alpha ^{-1}\left(
\left( \mu +\frac{\varepsilon }{2}\right) \beta \left( \log \gamma \left(
r_{n}\right) \right) \right) \right\} .
\end{equation*}%
Set $E_{2}=\bigcup\limits_{n=n_{1}}^{+\infty }\left[ \frac{n}{n+1}r_{n},r_{n}%
\right] .$ Then for $r\in E_{2}\subset \left( 1,+\infty \right) ,$ by using $%
\gamma (2r)\leq 2\gamma (r)$ and $\beta (r+O(1))=(1+o(1))\beta (r)$ as $%
r\rightarrow +\infty ,$ we obtain for any given $\varepsilon >0$%
\begin{equation*}
T\left( r,f\right) \leq T\left( r_{n},f\right) <\exp ^{\left[ 2\right]
}\left\{ \alpha ^{-1}\left( \left( \mu +\frac{\varepsilon }{2}\right) \beta
\left( \log \gamma \left( r_{n}\right) \right) \right) \right\}
\end{equation*}%
\begin{equation*}
\leq \exp ^{\left[ 2\right] }\left\{ \alpha ^{-1}\left( \left( \mu +\frac{%
\varepsilon }{2}\right) \beta \left( \log \gamma \left( \left( \frac{n+1}{n}%
\right) r\right) \right) \right) \right\}
\end{equation*}%
\begin{equation*}
\leq \exp ^{\left[ 2\right] }\left\{ \alpha ^{-1}\left( \left( \mu +\frac{%
\varepsilon }{2}\right) \beta \left( \log \gamma \left( 2r\right) \right)
\right) \right\}
\end{equation*}%
\begin{equation*}
\leq \exp ^{\left[ 2\right] }\left\{ \alpha ^{-1}\left( \left( \mu +\frac{%
\varepsilon }{2}\right) \beta \left( \log \left( 2\gamma \left( r\right)
\right) \right) \right) \right\}
\end{equation*}%
\begin{equation*}
=\exp ^{\left[ 2\right] }\left\{ \alpha ^{-1}\left( \left( \mu +\frac{%
\varepsilon }{2}\right) \beta \left( \log 2+\log \gamma \left( r\right)
\right) \right) \right\}
\end{equation*}%
\begin{equation*}
=\exp ^{\left[ 2\right] }\left\{ \alpha ^{-1}\left( \left( \mu +\frac{%
\varepsilon }{2}\right) \left( 1+o(1)\right) \beta \left( \log \gamma \left(
r\right) \right) \right) \right\}
\end{equation*}%
\begin{equation*}
<\exp ^{\left[ 2\right] }\left\{ \alpha ^{-1}\left( \left( \mu +\varepsilon
\right) \beta \left( \log \gamma \left( r\right) \right) \right) \right\} ,
\end{equation*}%
and $lm\left( E_{2}\right) =\sum\limits_{n=n_{1}}^{+\infty }\int\limits_{%
\frac{n}{n+1}r_{n}}^{r_{n}}\frac{dt}{t}=\sum\limits_{n=n_{1}}^{+\infty }\log
\left( 1+\frac{1}{n}\right) =\infty .$ Thus, Lemma \ref{l3.2} is proved.
\end{proof}

\qquad We can also prove the following result by using similar reason as in
the proof of Lemma \ref{l3.2}.

\begin{lemma}
\label{l3.3}\textbf{\ }Let $f$ be an entire function with $\mu _{(\alpha
,\beta ,\gamma )}[f]=\mu <+\infty $. \textit{Then there exists a set} $%
E_{3}\subset \left( 1,+\infty \right) $ \textit{with infinite logarithmic
measure such that for }$r\in E_{3}\subset \left( 1,+\infty \right) ,$
\textit{we have for any given }$\varepsilon >0$\textit{\ }%
\begin{equation*}
M\left( r,f\right) <\exp ^{\left[ 2\right] }\left\{ \alpha ^{-1}\left(
\left( \mu +\varepsilon \right) \beta \left( \log \gamma \left( r\right)
\right) \right) \right\} .
\end{equation*}
\end{lemma}

\qquad The following lemma gives the relation between the maximum term and
the central index of an entire finction $f$.

\begin{lemma}
\label{l3.4}(\cite{h}, Theorems 1.9 and 1.10, or \cite{gl}, Satz 4.3 and
4.4) \textit{Let }$f\left( z\right) =\sum\limits_{n=0}^{+\infty }a_{n}z^{n}$%
\textit{\ be an entire function, }$\mu \left( r\right) $\textit{\ be the
maximum term of }$f$\textit{, i.e., }%
\begin{equation*}
\mu \left( r\right) =\max \left\{ \left\vert a_{n}\right\vert
r^{n}:n=0,1,2,...\right\} ,
\end{equation*}%
\textit{and }$\nu \left( r,f\right) =\nu _{f}(r)$\textit{\ be the central
index of }$f$\textit{, i.e., }%
\begin{equation*}
\nu \left( r,f\right) =\max \left\{ m:\mu \left( r\right) =\left\vert
a_{m}\right\vert r^{m}\right\} \text{\textit{.}}
\end{equation*}%
\textit{Then}
\end{lemma}

\noindent $\left( \text{i}\right) $%
\begin{equation*}
\log \mu \left( r\right) =\log \left\vert a_{0}\right\vert +\underset{0}{%
\overset{r}{\int }}\frac{\nu _{f}(t)}{t}dt,
\end{equation*}%
\textit{here we assume that }$\left\vert a_{0}\right\vert \neq 0$\textit{.}

\noindent $\left( \text{ii}\right) $ \textit{For }$r<R$%
\begin{equation*}
M\left( r,f\right) <\mu \left( r\right) \left\{ \nu _{f}(R)+\frac{R}{R-r}%
\right\} .
\end{equation*}

\begin{lemma}
\label{l3.5}(\cite{17, gl, v}) \textit{Let }$f$\textit{\ be a transcendental
entire function. Then there exists a set }$E_{4}\subset (1,+\infty )$\textit{%
\ with finite logarithmic measure such that for all }$z$\textit{\ satisfying
}$\left\vert z\right\vert =r\notin E_{4}$\textit{\ and }$\left\vert
f(z)\right\vert =M(r,f),$\textit{\ we have}%
\begin{equation*}
\frac{f^{(n)}(z)}{f(z)}=\left( \frac{\nu _{f}(r)}{z}\right)
^{n}(1+o(1)),\quad (n\in
\mathbb{N}
).
\end{equation*}
\end{lemma}

\qquad Here, we give the generalized logarithmic derivative estimates for
meromorphic functions of finite $(\alpha (\log ),\beta ,\gamma )-$order$.$

\begin{lemma}
\label{l3.6} (\cite{b8}) Let $f$ be a meromorphic function of order $\rho
_{(\alpha (\log ),\beta ,\gamma )}[f]$ $=\rho <+\infty $, $k\in \mathbb{N}$.
Then, for any given $\varepsilon >0$,%
\begin{equation*}
m\left( r,\frac{f^{(k)}}{f}\right) =O\left( \exp \left\{ \alpha ^{-1}((\rho
+\varepsilon )\beta \left( \log \gamma \left( r\right) \right) )\right\}
\right) \text{,}
\end{equation*}%
outside, possibly, an exceptional set $E_{5}\subset \left[ 0,+\infty \right)
$ of finite linear measure.
\end{lemma}

\begin{lemma}
\label{l3.7} (\cite{b8}) Let $A_{0}(z),$ $A_{1}(z),...,A_{k-1}(z)$ be entire
functions. Then every nontrivial solution $f$ of $\left( \ref{1.1}\right) $
satisfies%
\begin{equation*}
\rho _{(\alpha (\log ),\beta ,\gamma )}[f]\leq \max \{\rho _{(\alpha ,\beta
,\gamma )}[A_{j}]:j=0,1,...,k-1\}\text{.}
\end{equation*}
\end{lemma}

\begin{lemma}
\label{l3.8} (\cite{b8}) Let $f$ be an entire function with $\rho _{(\alpha
,\beta ,\gamma )}[f]=\rho \in (0,+\infty )$ and $\tau _{(\alpha ,\beta
,\gamma ),M}[f]\in (0,+\infty )$. Then for any given $\eta <\tau _{(\alpha
,\beta ,\gamma ),M}[f]$, there exists a set $E_{6}\subset \left( 1,+\infty
\right) $ of infinite logarithmic measure such that for all $r\in E_{6}$,
one has%
\begin{equation*}
\exp \left\{ \alpha (\log ^{[2]}M(r,f))\right\} >\eta \left( \exp \left\{
\beta \left( \log \gamma \left( r\right) \right) \right\} \right) ^{\rho }%
\text{.}
\end{equation*}
\end{lemma}

\begin{lemma}
\label{l3.9} Let $f_{2}(z)$ be an entire function of lower $(\alpha \left(
\log \right) ,\beta ,\gamma )$-order with $\mu _{(\alpha \left( \log \right)
,\beta ,\gamma )}[f_{2}]=\mu >0$, and let $f_{1}(z)$ be an entire function
of $(\alpha \left( \log \right) ,\beta ,\gamma )$-order with $\rho _{(\alpha
\left( \log \right) ,\beta ,\gamma )}[f_{1}]=\rho <+\infty $. If $\rho
_{(\alpha \left( \log \right) ,\beta ,\gamma )}[f_{1}]<\mu _{(\alpha \left(
\log \right) ,\beta ,\gamma )}[f_{2}],$ then we have%
\begin{equation*}
T\left( r,f_{1}\right) =o(T\left( r,f_{2}\right) )\text{ as }r\rightarrow
+\infty \text{.}
\end{equation*}
\end{lemma}

\begin{proof}
By definitions of $(\alpha \left( \log \right) ,\beta ,\gamma )$-order and
lower $(\alpha \left( \log \right) ,\beta ,\gamma )$-order, for any given $%
\varepsilon $ with $0<2\varepsilon <\mu -\rho $ and sufficiently large $r$,
we have%
\begin{equation}
T(r,f_{1})\leq \exp ^{\left[ 2\right] }\left\{ \alpha ^{-1}\left( \left(
\rho +\varepsilon \right) \beta \left( \log \gamma \left( r\right) \right)
\right) \right\}  \label{3.1}
\end{equation}%
and%
\begin{equation}
T(r,f_{2})\geq \exp ^{\left[ 2\right] }\left\{ \alpha ^{-1}\left( \left( \mu
-\varepsilon \right) \beta \left( \log \gamma \left( r\right) \right)
\right) \right\} .  \label{3.2}
\end{equation}%
Now by (\ref{3.1}) and (\ref{3.2}), we get%
\begin{equation*}
\frac{T(r,f_{1})}{T(r,f)}\leq \frac{\exp ^{\left[ 2\right] }\left\{ \alpha
^{-1}\left( \left( \rho +\varepsilon \right) \beta \left( \log \gamma \left(
r\right) \right) \right) \right\} }{\exp ^{\left[ 2\right] }\left\{ \alpha
^{-1}\left( \left( \mu -\varepsilon \right) \beta \left( \log \gamma \left(
r\right) \right) \right) \right\} }
\end{equation*}%
\begin{equation*}
=\exp \left\{ \exp \left\{ \alpha ^{-1}\left( \left( \rho +\varepsilon
\right) \beta \left( \log \gamma \left( r\right) \right) \right) \right\}
-\exp \left\{ \alpha ^{-1}\left( \left( \mu -\varepsilon \right) \beta
\left( \log \gamma \left( r\right) \right) \right) \right\} \right\}
\end{equation*}%
\begin{equation*}
=\exp \left\{ \left( \frac{\exp \left\{ \alpha ^{-1}\left( \left( \rho
+\varepsilon \right) \beta \left( \log \gamma \left( r\right) \right)
\right) \right\} }{\exp \left\{ \alpha ^{-1}\left( \left( \mu -\varepsilon
\right) \beta \left( \log \gamma \left( r\right) \right) \right) \right\} }%
-1\right) \exp \left\{ \alpha ^{-1}\left( \left( \mu -\varepsilon \right)
\beta \left( \log \gamma \left( r\right) \right) \right) \right\} \right\}
\end{equation*}%
\begin{equation*}
=\exp \left\{ \left( \frac{\exp \left\{ \alpha ^{-1}\left( \frac{\rho
+\varepsilon }{\mu -\varepsilon }\left( \mu -\varepsilon \right) \beta
\left( \log \gamma \left( r\right) \right) \right) \right\} }{\exp \left\{
\alpha ^{-1}\left( \left( \mu -\varepsilon \right) \beta \left( \log \gamma
\left( r\right) \right) \right) \right\} }-1\right) \exp \left\{ \alpha
^{-1}\left( \left( \mu -\varepsilon \right) \beta \left( \log \gamma \left(
r\right) \right) \right) \right\} \right\} .
\end{equation*}%
Set%
\begin{equation*}
y=\left( \frac{\exp \left\{ \alpha ^{-1}\left( \frac{\rho +\varepsilon }{\mu
-\varepsilon }\left( \mu -\varepsilon \right) \beta \left( \log \gamma
\left( r\right) \right) \right) \right\} }{\exp \left\{ \alpha ^{-1}\left(
\left( \mu -\varepsilon \right) \beta \left( \log \gamma \left( r\right)
\right) \right) \right\} }-1\right) \exp \left\{ \alpha ^{-1}\left( \left(
\mu -\varepsilon \right) \beta \left( \log \gamma \left( r\right) \right)
\right) \right\} .
\end{equation*}%
Then by putting $\left( \mu -\varepsilon \right) \beta \left( \log \gamma
\left( r\right) \right) =x,$ $\frac{\rho +\varepsilon }{\mu -\varepsilon }=k$
$\left( 0<k<1\right) $ and making use of the condition $\alpha
^{-1}(kx)=o\left( \alpha ^{-1}(x)\right) $ $\left( 0<k<1\right) $ as $%
x\rightarrow +\infty ,$ we get
\begin{eqnarray*}
\lim_{r\rightarrow +\infty }y &=&\lim_{x\rightarrow +\infty }\left( \frac{%
\exp \left\{ \alpha ^{-1}\left( kx\right) \right\} }{\exp \left\{ \alpha
^{-1}\left( x\right) \right\} }-1\right) \exp \left\{ \alpha ^{-1}\left(
x\right) \right\} \\
&=&\lim_{x\rightarrow +\infty }\left( \frac{\exp \left\{ o\left( \alpha
^{-1}(x)\right) \right\} }{\exp \left\{ \alpha ^{-1}\left( x\right) \right\}
}-1\right) \exp \left\{ \alpha ^{-1}\left( x\right) \right\} \\
&=&\lim_{x\rightarrow +\infty }\left( \exp \left\{ \left( o\left( 1\right)
-1\right) \alpha ^{-1}\left( x\right) \right\} -1\right) \exp \left\{ \alpha
^{-1}\left( x\right) \right\} =-\infty ,
\end{eqnarray*}%
this implies%
\begin{equation*}
\lim_{r\rightarrow +\infty }\exp y=0.
\end{equation*}%
Therefore yielding%
\begin{equation*}
\underset{r\rightarrow +\infty }{\lim }\frac{T(r,f_{1})}{T(r,f_{2})}=0\text{,%
}
\end{equation*}%
that is $T\left( r,f_{1}\right) =o(T\left( r,f_{2}\right) )$ as $%
r\rightarrow +\infty $.
\end{proof}

\begin{lemma}
\label{l3.10} Let $F(z)\not\equiv 0$, $A_{j}(z)$ $(j=0,...,k-1)$ be
meromorphic functions, and let $f$ be a meromorphic solution of $\left( \ref%
{1.2}\right) $ satisfying
\begin{equation*}
\max \{\rho _{(\alpha \left( \log \right) ,\beta ,\gamma )}[A_{j}],\rho
_{(\alpha \left( \log \right) ,\beta ,\gamma )}[F]:j=0,1,...,k-1\}<\mu
_{(\alpha (\log ),\beta ,\gamma )}[f].
\end{equation*}%
Then we have%
\begin{equation*}
\underline{\overline{\lambda }}_{(\alpha (\log ),\beta ,\gamma )}[f]=%
\underline{\lambda }_{(\alpha (\log ),\beta ,\gamma )}[f]=\mu _{(\alpha
(\log ),\beta ,\gamma )}[f]\text{.}
\end{equation*}
\end{lemma}

\begin{proof}
By (\ref{1.2}), we get that%
\begin{equation}
\frac{1}{f}=\frac{1}{F}\left( \frac{f^{(k)}}{f}+A_{k-1}(z)\frac{f^{(k-1)}}{f}%
+\cdots +A_{1}(z)\frac{f^{\prime }}{f}+A_{0}\right) \text{.}  \label{3.3}
\end{equation}%
Now, by (\ref{1.2}) it is easy to see that if $f$ has a zero at $z_{0}$ of
order $a$ $(a>k)$, and if $A_{0},...,A_{k-1}$ are analytic at $z_{0}$, then $%
F(z)$ must have a zero at $z_{0}$ of order $a-k$, hence%
\begin{equation}
n\left( r,\frac{1}{f}\right) \leq k\overline{n}\left( r,\frac{1}{f}\right)
+n\left( r,\frac{1}{F}\right) +\sum_{j=0}^{k-1}n(r,A_{j})  \label{3.4}
\end{equation}%
and%
\begin{equation}
N\left( r,\frac{1}{f}\right) \leq k\overline{N}\left( r,\frac{1}{f}\right)
+N\left( r,\frac{1}{F}\right) +\sum_{j=0}^{k-1}N(r,A_{j})\text{.}
\label{3.5}
\end{equation}%
By the lemma on logarithmic derivative (\cite{1}, p. 34) and (\ref{3.3}), we
have%
\begin{equation}
m\left( r,\frac{1}{f}\right) \leq m\left( r,\frac{1}{F}\right)
+\sum_{j=0}^{k-1}m(r,A_{j})+O(\log T(r,f)+\log r)\text{ \ \ }(r\notin E_{5})%
\text{,}  \label{3.6}
\end{equation}%
where $E_{5}$ is a set of $r$ of finite linear measure. By (\ref{3.5}) and (%
\ref{3.6}), we obtain that%
\begin{equation*}
T(r,f)=T\left( r,\frac{1}{f}\right) +O(1)\leq k\overline{N}\left( r,\frac{1}{%
f}\right)
\end{equation*}%
\begin{equation}
+T(r,F)+\sum_{j=0}^{k-1}T(r,A_{j})+O(\log (rT(r,f)))\text{ }(r\notin E_{5})%
\text{.}  \label{3.7}
\end{equation}%
Since $\max \{\rho _{(\alpha (\log ),\beta ,\gamma )}[A_{j}],\rho _{(\alpha
(\log ),\beta ,\gamma )}[F]:j=0,1,...,k-1\}<\mu _{(\alpha (\log ),\beta
,\gamma )}[f]$, then by Lemma \ref{l3.9}
\begin{equation}
T\left( r,F\right) =o(T\left( r,f\right) ),\text{ }T\left( r,A_{j}\right)
=o(T\left( r,f\right) )\text{ }(j=0,...,k-1)\text{ as }r\rightarrow +\infty
\text{.}  \label{3.8}
\end{equation}%
Since $f$ is transcendental, then we have%
\begin{equation}
O(\log (rT(r,f)))=o(T(r,f))\text{ as }r\rightarrow +\infty \text{.}
\label{3.9}
\end{equation}%
Therefore, by substituting (\ref{3.8}) and (\ref{3.9}) into (\ref{3.7}), for
all $|z|=r\notin E_{5}$, we get that%
\begin{equation*}
T(r,f)\leq O\left( \overline{N}\left( r,\frac{1}{f}\right) \right) \text{.}
\end{equation*}%
Hence from above we have%
\begin{equation*}
\mu _{(\alpha (\log ),\beta ,\gamma )}[f]\leq \underline{\overline{\lambda }}%
_{(\alpha (\log ),\beta ,\gamma )}[f]\text{.}
\end{equation*}%
Since $\underline{\overline{\lambda }}_{(\alpha (\log ),\beta ,\gamma
)}[f]\leq \underline{\lambda }_{(\alpha (\log ),\beta ,\gamma )}[f]\leq \mu
_{(\alpha (\log ),\beta ,\gamma )}[f],$ then%
\begin{equation*}
\underline{\overline{\lambda }}_{(\alpha (\log ),\beta ,\gamma )}[f]=%
\underline{\lambda }_{(\alpha (\log ),\beta ,\gamma )}[f]=\mu _{(\alpha
(\log ),\beta ,\gamma )}[f]\text{.}
\end{equation*}
\end{proof}

\begin{lemma}
\label{l3.11} (\cite{b7}) Let $F(z)\not\equiv 0$, $A_{j}(z)$ $(j=0,...,k-1)$
be entire functions. Also let $f$ be a solution of $\left( \ref{1.2}\right) $
satisfying $\max \{\rho _{(\alpha (\log ),\beta ,\gamma )}[A_{j}],\rho
_{(\alpha (\log ),\beta ,\gamma )}[F]:j=0,1,...,k-1\}<\rho _{(\alpha (\log
),\beta ,\gamma )}[f]$. Then we have%
\begin{equation*}
\overline{\lambda }_{(\alpha (\log ),\beta ,\gamma )}[f]=\lambda _{(\alpha
(\log ),\beta ,\gamma )}[f]=\rho _{(\alpha (\log ),\beta ,\gamma )}[f]\text{.%
}
\end{equation*}
\end{lemma}

\begin{lemma}
\label{l3.12} Let $f$ be a transcendental entire function. Then $\rho
_{(\alpha (\log ),\beta ,\gamma )}[f]$ $=\rho _{(\alpha (\log ),\beta
,\gamma )}[f^{\left( k\right) }],$ $k\in
\mathbb{N}
$.
\end{lemma}

\begin{proof}
By Lemma 4.4 in (\cite{b7}), we have $\rho _{(\alpha (\log ),\beta ,\gamma
)}[f]$ $=\rho _{(\alpha (\log ),\beta ,\gamma )}[f^{\prime }],$ so by using
mathematical induction, we easily obtain the result.
\end{proof}

\begin{lemma}
\label{l3.13} (\cite{b7}) \textit{Let }$f$\textit{\ be a meromorphic
function. If }$\rho _{(\alpha ,\beta ,\gamma )}[f]=\rho <+\infty $\textit{,
then }$\rho _{(\alpha \left( \log \right) ,\beta ,\gamma )}[f]=0.$
\end{lemma}

\begin{lemma}
\label{l3.14} (\cite{h}) Let $A_{j}(z)$ $(j=0,...,k-1)$ be entire
coefficients in $\left( \ref{1.1}\right) $, and at least one of them is
transcendental. If $A_{s}(z)$ $\left( 0\leq s\leq k-1\right) $ is the first
one (according to the sequence of $A_{0}(z),...,A_{k-1}(z))$ satisfying
\begin{equation*}
\underset{r\rightarrow +\infty }{\lim \inf }\frac{\overset{k-1}{%
\sum\limits_{j=s+1}}m\left( r,A_{j}\right) }{m\left( r,A_{s}\right) }<1,%
\text{ }r\notin E_{6},
\end{equation*}%
where $E_{6}$ is a set of $r$ of finite linear measure. Then $\left( \ref%
{1.1}\right) $ possesses at most $s$ linearly independent entire solutions
satisfying
\begin{equation*}
\underset{r\rightarrow +\infty }{\lim \sup }\frac{\log T\left( r,f\right) }{%
m\left( r,A_{s}\right) }=0,\text{ }r\notin E_{6}.
\end{equation*}
\end{lemma}

\section{\textbf{Proof of the Main Results}}

\textbf{Proof of Theorem \ref{t2.3}}. Suppose that $f$\ $\left( \not\equiv
0\right) $ is a solution of equation $\left( \ref{1.1}\right) $. By Theorem %
\ref{t2.1}, we know that every solution $f$\ $\left( \not\equiv 0\right) $
of $\left( \ref{1.1}\right) $ satisfies \textit{\ }$\rho _{(\alpha (\log
),\beta ,\gamma )}[f]=\rho _{(\alpha ,\beta ,\gamma )}[A_{0}]$. So, we only
need to prove that every solution $f$\ $\left( \not\equiv 0\right) $ of $%
\left( \ref{1.1}\right) $ satisfies \textit{\ }$\mu _{(\alpha (\log ),\beta
,\gamma )}[f]=\mu _{(\alpha ,\beta ,\gamma )}[A_{0}].$ First, we prove that $%
\mu _{1}=\mu _{(\alpha (\log ),\beta ,\gamma )}[f]\geq \mu _{(\alpha ,\beta
,\gamma )}[A_{0}]=\mu _{0}$. Suppose the contrary. Set $\max \{\rho
_{(\alpha ,\beta ,\gamma )}[A_{j}]:j=1,...,k-1,\mu _{(\alpha (\log ),\beta
,\gamma )}[f]\}=\rho <\mu _{(\alpha ,\beta ,\gamma )}[A_{0}]=\mu _{0}.$ From
$\left( \ref{1.1}\right) ,$ we can write%
\begin{equation}
\left\vert A_{0}\left( z\right) \right\vert \leq \left\vert \frac{f^{\left(
k\right) }}{f}\right\vert +\left\vert A_{k-1}\left( z\right) \right\vert
\left\vert \frac{f^{\left( k-1\right) }}{f}\right\vert +\cdots +\left\vert
A_{1}\left( z\right) \right\vert \left\vert \frac{f^{\prime }}{f}\right\vert
.  \label{4.1}
\end{equation}%
For any given $\varepsilon \left( 0<2\varepsilon <\mu _{0}-\rho \right) $
and for sufficiently large $r$, we have%
\begin{equation}
\left\vert A_{0}\left( z\right) \right\vert >\exp ^{[2]}\left\{ \alpha
^{-1}\left( \left( \mu _{0}-\varepsilon \right) \beta \left( \log \gamma
\left( r\right) \right) \right) \right\}  \label{4.2}
\end{equation}%
and
\begin{equation}
\left\vert A_{j}\left( z\right) \right\vert \leq \exp ^{[2]}\left\{ \alpha
^{-1}\left( \left( \rho +\frac{\varepsilon }{2}\right) \beta \left( \log
\gamma \left( r\right) \right) \right) \right\} ,\text{ }j\in \left\{
1,2,...,k-1\right\} .  \label{4.3}
\end{equation}%
By Lemma \ref{l3.1}, there exist a constant $B>0$ and a set $E_{1}\subset
\left( 1,+\infty \right) $ having finite logarithmic measure such that for
all $z$ satisfying $\left\vert z\right\vert =r\notin \left[ 0,1\right] \cup
E_{1}$, we have%
\begin{equation}
\left\vert \frac{f^{\left( j\right) }(z)}{f(z)}\right\vert \leq B\left[
T(2r,f)\right] ^{k+1}\text{ }\left( j=1,2,...,k\right) .  \label{4.4}
\end{equation}%
It follows by Lemma \ref{l3.2} and $\left( \ref{4.4}\right) $, that for
sufficiently large $\left\vert z\right\vert =r\in E_{2}\backslash \left(
E_{1}\cup \left[ 0,1\right] \right) $%
\begin{equation*}
\left\vert \frac{f^{\left( j\right) }(z)}{f(z)}\right\vert \leq B\left[
T(2r,f)\right] ^{k+1}
\end{equation*}%
\begin{equation}
\leq B\left[ \exp ^{[2]}\left\{ \alpha ^{-1}\left( \left( \mu _{1}+\frac{%
\varepsilon }{2}\right) \beta \left( \log \gamma \left( r\right) \right)
\right) \right\} \right] ^{k+1}\text{ }\left( j=1,2,...,k\right) ,
\label{4.5}
\end{equation}%
where $E_{2}$ is a set of infinite logarithmic measure. Hence, by
substituting $\left( \ref{4.2}\right) -\left( \ref{4.5}\right) $ into $%
\left( \ref{4.1}\right) $, for the above $\varepsilon \left( 0<2\varepsilon
<\mu _{0}-\rho \right) $, we obtain for sufficiently large $\left\vert
z\right\vert =r\in E_{2}\backslash \left( E_{1}\cup \left[ 0,1\right]
\right) $%
\begin{equation*}
\exp ^{[2]}\left\{ \alpha ^{-1}\left( \left( \mu _{0}-\varepsilon \right)
\beta \left( \log \gamma \left( r\right) \right) \right) \right\}
\end{equation*}%
\begin{equation*}
\leq Bk\exp ^{[2]}\left\{ \alpha ^{-1}\left( \left( \rho +\frac{\varepsilon
}{2}\right) \beta \left( \log \gamma \left( r\right) \right) \right)
\right\} \left[ T(2r,f)\right] ^{k+1}
\end{equation*}%
\begin{equation*}
\leq Bk\exp ^{[2]}\left\{ \alpha ^{-1}\left( \left( \rho +\frac{\varepsilon
}{2}\right) \beta \left( \log \gamma \left( r\right) \right) \right) \right\}
\end{equation*}%
\begin{equation*}
\times \left[ \exp ^{[2]}\left\{ \alpha ^{-1}\left( \left( \mu _{1}+\frac{%
\varepsilon }{2}\right) \beta \left( \log \gamma \left( r\right) \right)
\right) \right\} \right] ^{k+1}
\end{equation*}%
\begin{equation}
\leq \exp ^{[2]}\left\{ \alpha ^{-1}\left( \left( \rho +\varepsilon \right)
\beta \left( \log \gamma \left( r\right) \right) \right) \right\} .
\label{4.6}
\end{equation}%
Since $E_{2}\backslash \left( E_{1}\cup \left[ 0,1\right] \right) $ is a set
of infinite logarithmic measure, then there exists a sequence of points $%
\left\vert z_{n}\right\vert =r_{n}\in E_{2}\backslash \left( E_{1}\cup \left[
0,1\right] \right) $ tending to $+\infty .$ It follows by $\left( \ref{4.6}%
\right) $ that%
\begin{equation}
\exp ^{[2]}\left\{ \alpha ^{-1}\left( \left( \mu _{0}-\varepsilon \right)
\beta \left( \log \gamma \left( r_{n}\right) \right) \right) \right\} \leq
\exp ^{[2]}\left\{ \alpha ^{-1}\left( \left( \rho +\varepsilon \right) \beta
\left( \log \gamma \left( r_{n}\right) \right) \right) \right\}  \label{4.7}
\end{equation}%
holds for all $z_{n}$ satisfying $\left\vert z_{n}\right\vert =r_{n}\in
E_{2}\backslash \left( E_{1}\cup \left[ 0,1\right] \right) $ as $\left\vert
z_{n}\right\vert \rightarrow +\infty .$ By arbitrariness of $\varepsilon >0$
and the monotony of the function $\alpha ^{-1}$, from $\left( \ref{4.7}%
\right) $ we obtain that $\rho \geq $ $\mu _{(\alpha ,\beta ,\gamma
)}[A_{0}]=\mu _{0}$. This contradiction proves the inequality $\mu _{(\alpha
(\log ),\beta ,\gamma )}[f]\geq \mu _{(\alpha ,\beta ,\gamma )}[A_{0}]$.%
\newline
\qquad Now, we prove $\mu _{(\alpha (\log ),\beta ,\gamma )}[f]\leq \mu
_{(\alpha ,\beta ,\gamma )}[A_{0}]=\mu _{0}.$ By $\left( \ref{1.1}\right) ,$
we have%
\begin{equation}
\left\vert \frac{f^{\left( k\right) }}{f}\right\vert \leq \left\vert
A_{k-1}\left( z\right) \right\vert \left\vert \frac{f^{\left( k-1\right) }}{f%
}\right\vert +\cdots +\left\vert A_{1}\left( z\right) \right\vert \left\vert
\frac{f^{\prime }}{f}\right\vert +\left\vert A_{0}\left( z\right)
\right\vert .  \label{4.8}
\end{equation}%
By Lemma \ref{l3.5}, there exists a set $E_{4}\subset \left( 1,+\infty
\right) $\ of finite logarithmic measure such that the estimation%
\begin{equation}
\frac{f^{\left( j\right) }(z)}{f(z)}=\left( \frac{\nu _{f}\left( r\right) }{z%
}\right) ^{j}\left( 1+o\left( 1\right) \right) \text{ }\left(
j=1,...,k\right)  \label{4.9}
\end{equation}%
holds for all $z$ satisfying $\left\vert z\right\vert =r\notin E_{4},$ $%
r\rightarrow +\infty $\ and $\left\vert f\left( z\right) \right\vert
=M\left( r,f\right) $. By Lemma \ref{l3.3}, for any given $\varepsilon >0,$\
there exists a set $E_{3}\subset \left( 1,+\infty \right) $ that has
infinite logarithmic measure, such that%
\begin{equation}
\left\vert A_{0}\left( z\right) \right\vert \leq \exp ^{[2]}\left\{ \alpha
^{-1}\left( \left( \mu _{0}+\frac{\varepsilon }{2}\right) \beta \left( \log
\gamma \left( r\right) \right) \right) \right\}  \label{4.10}
\end{equation}%
and for sufficiently large $r$%
\begin{equation*}
\left\vert A_{j}\left( z\right) \right\vert \leq \exp ^{[2]}\left\{ \alpha
^{-1}\left( \left( \rho +\frac{\varepsilon }{2}\right) \beta \left( \log
\gamma \left( r\right) \right) \right) \right\}
\end{equation*}%
\begin{equation}
\leq \exp ^{[2]}\left\{ \alpha ^{-1}\left( \left( \mu _{0}+\frac{\varepsilon
}{2}\right) \beta \left( \log \gamma \left( r\right) \right) \right)
\right\} \text{ }\left( j=1,...,k-1\right) .  \label{4.11}
\end{equation}%
Substituting $\left( \ref{4.9}\right) ,$ $\left( \ref{4.10}\right) $ and $%
\left( \ref{4.11}\right) $ into $\left( \ref{4.8}\right) ,$ we obtain%
\begin{equation*}
\nu _{f}\left( r\right) \leq kr^{k}\left\vert 1+o\left( 1\right) \right\vert
\exp ^{[2]}\left\{ \alpha ^{-1}\left( \left( \mu _{0}+\frac{\varepsilon }{2}%
\right) \beta \left( \log \gamma \left( r\right) \right) \right) \right\}
\end{equation*}%
\begin{equation}
\leq \exp ^{[2]}\left\{ \alpha ^{-1}\left( \left( \mu _{0}+\varepsilon
\right) \beta \left( \log \gamma \left( r\right) \right) \right) \right\}
\label{4.12}
\end{equation}%
for all $z$ satisfying $\left\vert z\right\vert =r\in E_{3}\backslash E_{4},$
$r\rightarrow +\infty $ and $\left\vert f\left( z\right) \right\vert
=M\left( r,f\right) .$ By Lemma \ref{l3.4}, from $\left( \ref{4.12}\right) $
we obtain for each $\varepsilon >0$%
\begin{equation*}
T\left( r,f\right) \leq \log M\left( r,f\right) <\log \left[ \mu \left(
r\right) \left( \nu _{f}\left( 2r\right) +2\right) \right]
\end{equation*}%
\begin{equation*}
=\log \left[ \left\vert a_{\nu _{f}\left( r\right) }\right\vert r^{\nu
_{f}\left( r\right) }\left( \nu _{f}\left( 2r\right) +2\right) \right] <\nu
_{f}\left( r\right) \log r+\log \left( 2\nu _{f}\left( 2r\right) \right)
+\log \left\vert a_{\nu _{f}\left( r\right) }\right\vert
\end{equation*}%
\begin{equation*}
\leq \exp ^{[2]}\left\{ \alpha ^{-1}\left( \left( \mu _{0}+\varepsilon
\right) \beta \left( \log \gamma \left( r\right) \right) \right) \right\}
\log r
\end{equation*}%
\begin{equation*}
+\log \left( 2\exp ^{[2]}\left\{ \alpha ^{-1}\left( \left( \mu
_{0}+\varepsilon \right) \beta \left( \log \gamma \left( 2r\right) \right)
\right) \right\} \right) +\log \left\vert a_{\nu _{f}\left( r\right)
}\right\vert
\end{equation*}%
\begin{equation*}
\leq \exp ^{[2]}\left\{ \alpha ^{-1}\left( \left( \mu _{0}+2\varepsilon
\right) \beta \left( \log \gamma \left( r\right) \right) \right) \right\}
+\log 2
\end{equation*}%
\begin{equation*}
+\exp \left\{ \alpha ^{-1}\left( \left( \mu _{0}+\varepsilon \right) \beta
\left( \log \gamma \left( 2r\right) \right) \right) \right\} +\log
\left\vert a_{\nu _{f}\left( r\right) }\right\vert
\end{equation*}%
\begin{equation*}
\leq \exp ^{[2]}\left\{ \alpha ^{-1}\left( \left( \mu _{0}+3\varepsilon
\right) \beta \left( \log \gamma \left( r\right) \right) \right) \right\} .
\end{equation*}%
Hence,%
\begin{equation*}
\frac{\alpha (\log ^{[2]}T(r,f))}{\beta \left( \log \gamma \left( r\right)
\right) }\leq \mu _{0}+3\varepsilon .
\end{equation*}%
It follows
\begin{equation*}
\mu _{(\alpha (\log ),\beta ,\gamma )}[f]=\underset{r\longrightarrow +\infty
}{\lim \inf }\frac{\alpha (\log ^{[2]}T(r,f))}{\beta \left( \log \gamma
\left( r\right) \right) }\leq \mu _{0}+3\varepsilon .
\end{equation*}%
Since $\varepsilon >0$ is arbitrary, then we obtain $\mu _{(\alpha (\log
),\beta ,\gamma )}[f]\leq \mu _{0}.$ Hence every solution $f\not\equiv 0$ of
equation\textit{\ }$\left( \ref{1.1}\right) \ $satisfies $\mu _{(\alpha
,\beta ,\gamma )}[A_{0}]=\mu _{(\alpha (\log ),\beta ,\gamma )}[f]\leq \rho
_{(\alpha (\log ),\beta ,\gamma )}[f]=\rho _{(\alpha ,\beta ,\gamma
)}[A_{0}].$\newline
\qquad Secondly, we prove that $\underline{\overline{\lambda }}_{(\alpha
(\log ),\beta ,\gamma )}[f-g]=\mu _{(\alpha (\log ),\beta ,\gamma )}[f]$ and
\begin{equation*}
\overline{\lambda }_{(\alpha (\log ),\beta ,\gamma )}[f-g]=\rho _{(\alpha
(\log ),\beta ,\gamma )}[f].
\end{equation*}%
Set $h=f-g.$ Since
\begin{equation*}
\rho _{(\alpha (\log ),\beta ,\gamma )}\left[ g\right] <\mu _{(\alpha ,\beta
,\gamma )}[A_{0}]=\mu _{(\alpha (\log ),\beta ,\gamma )}[f]\leq \rho
_{(\alpha (\log ),\beta ,\gamma )}[f],
\end{equation*}%
it follows from Proposition \ref{p1.2} and Proposition \ref{p1.4} that $\rho
_{(\alpha (\log ),\beta ,\gamma )}\left[ h\right] =\rho _{(\alpha (\log
),\beta ,\gamma )}[f]=\rho _{(\alpha ,\beta ,\gamma )}[A_{0}]$ and $\mu
_{(\alpha (\log ),\beta ,\gamma )}[h]=\mu _{(\alpha (\log ),\beta ,\gamma
)}[f]=\mu _{(\alpha ,\beta ,\gamma )}[A_{0}].$ By substituting $f=g+h,$ $%
f^{\prime }=g^{\prime }+h^{\prime },\ldots ,f^{\left( k\right)
}=g^{(k)}+h^{(k)}$ into $\left( 1.1\right) ,$ we obtain
\begin{equation}
h^{(k)}+A_{k-1}(z)h^{(k-1)}+\cdots
+A_{0}(z)h=-(g^{(k)}+A_{k-1}(z)g^{(k-1)}+\cdots +A_{0}(z)g).  \label{4.13}
\end{equation}%
If $g^{(k)}+A_{k-1}(z)g^{(k-1)}+\cdots +A_{0}(z)g=G\equiv 0,$ then by the
first part of the proof of Theorem \ref{t2.3} we have $\rho _{(\alpha (\log
),\beta ,\gamma )}[g]\geq \mu _{(\alpha ,\beta ,\gamma )}[A_{0}]$ which
contradicts the assumption $\rho _{(\alpha (\log ),\beta ,\gamma )}[g]<\mu
_{(\alpha ,\beta ,\gamma )}[A_{0}].$ Hence $G\not\equiv 0.$ By Proposition %
\ref{p1.2}, Lemma \ref{l3.12} and Lemma \ref{l3.13}$,$ we get%
\begin{equation}
\rho _{(\alpha (\log ),\beta ,\gamma )}\left[ G\right] \leq \max \{\rho
_{(\alpha (\log ),\beta ,\gamma )}\left[ g\right] ,\rho _{(\alpha (\log
),\beta ,\gamma )}(A_{j})\text{ }\left( j=0,1,...,k-1\right) \}  \notag
\end{equation}%
\begin{equation*}
=\rho _{(\alpha (\log ),\beta ,\gamma )}\left[ g\right] <\mu _{(\alpha
,\beta ,\gamma )}[A_{0}]=\mu _{(\alpha (\log ),\beta ,\gamma )}[f]=\mu
_{(\alpha (\log ),\beta ,\gamma )}[h]
\end{equation*}%
\begin{equation*}
\leq \rho _{(\alpha (\log ),\beta ,\gamma )}\left[ h\right] =\rho _{(\alpha
(\log ),\beta ,\gamma )}[f]=\rho _{(\alpha ,\beta ,\gamma )}[A_{0}].
\end{equation*}%
Then, it follows from Lemma \ref{l3.10}, Lemma \ref{l3.11} and $\left( \ref%
{4.13}\right) $ that $\overline{\lambda }_{(\alpha (\log ),\beta ,\gamma )}%
\left[ h\right] =\lambda _{(\alpha (\log ),\beta ,\gamma )}\left[ h\right]
=\rho _{(\alpha (\log ),\beta ,\gamma )}(h)=\rho _{(\alpha (\log ),\beta
,\gamma )}[f]$ and
\begin{equation*}
\underline{\overline{\lambda }}_{(\alpha (\log ),\beta ,\gamma )}[h]=%
\underline{\lambda }_{(\alpha (\log ),\beta ,\gamma )}[h]=\mu _{(\alpha
(\log ),\beta ,\gamma )}[h]=\mu _{(\alpha (\log ),\beta ,\gamma )}[f].
\end{equation*}%
Therefore, $\underline{\overline{\lambda }}_{(\alpha (\log ),\beta ,\gamma
)}[f-g]=\mu _{(\alpha (\log ),\beta ,\gamma )}[f]$ and
\begin{equation*}
\overline{\lambda }_{(\alpha (\log ),\beta ,\gamma )}\left[ f-g\right] =\rho
_{(\alpha (\log ),\beta ,\gamma )}[f]
\end{equation*}%
which completes the proof of Theorem \ref{t2.3}.

\textbf{Proof of Theorem \ref{t2.4}}. Suppose that $f$\ $\left( \not\equiv
0\right) $ is a solution of equation $\left( \ref{1.1}\right) $. Then by
Theorem \ref{t2.2}, we obtain $\rho _{(\alpha (\log ),\beta ,\gamma
)}[f]=\rho _{(\alpha ,\beta ,\gamma )}[A_{0}]$\textit{. }Now, we prove that $%
\mu _{1}=\mu _{(\alpha (\log ),\beta ,\gamma )}[f]\geq \mu _{(\alpha ,\beta
,\gamma )}[A_{0}]=\mu _{0}$. Suppose the contrary $\mu _{1}=\mu _{(\alpha
(\log ),\beta ,\gamma )}[f]<\mu _{(\alpha ,\beta ,\gamma )}[A_{0}]=\mu _{0}$%
. We set $b=\max \{\rho _{(\alpha ,\beta ,\gamma )}[A_{j}]:\rho _{(\alpha
,\beta ,\gamma )}[A_{j}]<\mu _{(\alpha ,\beta ,\gamma )}[A_{0}]\}.$ If $\rho
_{(\alpha ,\beta ,\gamma )}[A_{j}]<\mu _{(\alpha ,\beta ,\gamma )}[A_{0}],$
then for any given $\varepsilon $ with $0<3\varepsilon <\min \left\{ \mu
_{0}-b,\tau -\tau _{1}\right\} $ and for sufficiently large $r$, we have%
\begin{equation*}
\left\vert A_{j}\left( z\right) \right\vert \leq \exp ^{[2]}\left\{ \alpha
^{-1}\left( \left( b+\varepsilon \right) \beta \left( \log \gamma \left(
r\right) \right) \right) \right\}
\end{equation*}%
\begin{equation}
\leq \exp ^{[2]}\left\{ \alpha ^{-1}\left( \left( \mu _{(\alpha ,\beta
,\gamma )}[A_{0}]-2\varepsilon \right) \beta \left( \log \gamma \left(
r\right) \right) \right) \right\} .  \label{4.14}
\end{equation}%
If $\rho _{(\alpha ,\beta ,\gamma )}[A_{j}]=\mu _{(\alpha ,\beta ,\gamma
)}[A_{0}]$, $\tau _{(\alpha ,\beta ,\gamma ),M}\left[ A_{j}\right] \leq \tau
_{1}<\underline{\tau }_{(\alpha ,\beta ,\gamma ),M}[A_{0}]=\tau ,$ then for
sufficiently large $r$, we have%
\begin{equation}
\left\vert A_{j}\left( z\right) \right\vert \leq \exp ^{[2]}\left\{ \alpha
^{-1}\left( \log \left( \left( \tau _{1}+\varepsilon \right) \left( \exp
\left\{ \beta \left( \log \gamma \left( r\right) \right) \right\} \right)
^{\mu _{0}}\right) \right) \right\}  \label{4.15}
\end{equation}%
and%
\begin{equation}
\left\vert A_{0}\left( z\right) \right\vert >\exp ^{[2]}\left\{ \alpha
^{-1}\left( \log \left( \left( \tau -\varepsilon \right) \left( \exp \left\{
\beta \left( \log \gamma \left( r\right) \right) \right\} \right) ^{\mu
_{0}}\right) \right) \right\} .  \label{4.16}
\end{equation}%
By Lemma \ref{l3.1} and Lemma \ref{l3.2}, for any given $\varepsilon $ with $%
0<\varepsilon <\mu _{0}-\mu _{1}$ and sufficiently large $\left\vert
z\right\vert =r\in E_{2}\backslash \left( E_{1}\cup \left[ 0,1\right]
\right) $%
\begin{equation*}
\left\vert \frac{f^{\left( j\right) }(z)}{f(z)}\right\vert \leq B\left[
T(2r,f)\right] ^{k+1}
\end{equation*}%
\begin{equation}
\leq B\left[ \exp ^{[2]}\left\{ \alpha ^{-1}\left( \left( \mu
_{1}+\varepsilon \right) \beta \left( \log \gamma \left( r\right) \right)
\right) \right\} \right] ^{k+1}\text{ }\left( j=1,2,...,k\right) ,
\label{4.17}
\end{equation}%
where $E_{2}$ is a set of infinite logarithmic measure. Hence, by
substituting $\left( \ref{4.14}\right) -\left( \ref{4.17}\right) $ into $%
\left( \ref{4.1}\right) $, for the above $\varepsilon $ with $\
0<\varepsilon <\min \left\{ \frac{\mu _{0}-b}{3},\frac{\tau -\tau _{1}}{3}%
,\mu _{0}-\mu _{1}\right\} ,$ we obtain for sufficiently large $\left\vert
z\right\vert =r\in E_{2}\backslash \left( E_{1}\cup \left[ 0,1\right]
\right) $%
\begin{equation*}
\exp ^{[2]}\left\{ \alpha ^{-1}\left( \log \left( \left( \tau -\varepsilon
\right) \left( \exp \left\{ \beta \left( \log \gamma \left( r\right) \right)
\right\} \right) ^{\mu _{0}}\right) \right) \right\}
\end{equation*}%
\begin{equation*}
\leq Bk\exp ^{[2]}\left\{ \alpha ^{-1}\left( \log \left( \left( \tau
_{1}+\varepsilon \right) \left( \exp \left\{ \beta \left( \log \gamma \left(
r\right) \right) \right\} \right) ^{\mu _{0}}\right) \right) \right\} \left[
T(2r,f)\right] ^{k+1}
\end{equation*}%
\begin{equation*}
\leq Bk\exp ^{[2]}\left\{ \alpha ^{-1}\left( \log \left( \left( \tau
_{1}+\varepsilon \right) \left( \exp \left\{ \beta \left( \log \gamma \left(
r\right) \right) \right\} \right) ^{\mu _{0}}\right) \right) \right\}
\end{equation*}%
\begin{equation*}
\times \left[ \exp ^{[2]}\left\{ \alpha ^{-1}\left( \left( \mu
_{1}+\varepsilon \right) \beta \left( \log \gamma \left( r\right) \right)
\right) \right\} \right] ^{k+1}
\end{equation*}%
\begin{equation}
\leq \exp ^{[2]}\left\{ \alpha ^{-1}\left( \log \left( \left( \tau
_{1}+2\varepsilon \right) \left( \exp \left\{ \beta \left( \log \gamma
\left( r\right) \right) \right\} \right) ^{\mu _{0}}\right) \right) \right\}
.  \label{4.18}
\end{equation}%
Since $E_{2}\backslash \left( E_{1}\cup \left[ 0,1\right] \right) $ is a set
of infinite logarithmic measure, then there exists a sequence of points $%
\left\vert z_{n}\right\vert =r_{n}\in E_{2}\backslash \left( E_{1}\cup \left[
0,1\right] \right) $ tending to $+\infty .$ It follows by $\left( \ref{4.18}%
\right) $ that%
\begin{equation*}
\exp ^{[2]}\left\{ \alpha ^{-1}\left( \log \left( \left( \tau -\varepsilon
\right) \left( \exp \left\{ \beta \left( \log \gamma \left( r_{n}\right)
\right) \right\} \right) ^{\mu _{0}}\right) \right) \right\}
\end{equation*}%
\begin{equation*}
\leq \exp ^{[2]}\left\{ \alpha ^{-1}\left( \log \left( \left( \tau
_{1}+2\varepsilon \right) \left( \exp \left\{ \beta \left( \log \gamma
\left( r_{n}\right) \right) \right\} \right) ^{\mu _{0}}\right) \right)
\right\}
\end{equation*}%
holds for all $z_{n}$ satisfying $\left\vert z_{n}\right\vert =r_{n}\in
E_{2}\backslash \left( E_{1}\cup \left[ 0,1\right] \right) $ as $\left\vert
z_{n}\right\vert \rightarrow +\infty .$ By arbitrariness of $\varepsilon >0$
and the monotonicity of the function $\alpha ^{-1}$, we obtain that $\tau
_{1}\geq \tau $. This contradiction proves the inequality $\mu _{(\alpha
(\log ),\beta ,\gamma )}[f]\geq \mu _{(\alpha ,\beta ,\gamma )}[A_{0}]$.%
\newline
\qquad Now, we prove $\mu _{(\alpha (\log ),\beta ,\gamma )}[f]\leq \mu
_{(\alpha ,\beta ,\gamma )}[A_{0}].$ By using similar arguments as in the
proofs of Theorem \ref{t2.3}, we obtain $\mu _{(\alpha (\log ),\beta ,\gamma
)}[f]\leq \mu _{(\alpha ,\beta ,\gamma )}[A_{0}].$ Hence, every solution $%
f\not\equiv 0$ of equation\textit{\ }$\left( \ref{1.1}\right) \ $satisfies
\begin{equation*}
\mu _{(\alpha ,\beta ,\gamma )}[A_{0}]=\mu _{(\alpha (\log ),\beta ,\gamma
)}[f]\leq \rho _{(\alpha (\log ),\beta ,\gamma )}[f]=\rho _{(\alpha ,\beta
,\gamma )}[A_{0}].
\end{equation*}%
The second part of the proof of Theorem \ref{t2.3} completes the proof of
Theorem \ref{t2.4}.

\textbf{Proof of Theorem \ref{t2.5}}. Suppose that $f$\ $\left( \not\equiv
0\right) $ is a solution of equation $\left( \ref{1.1}\right) $. We divide
the proof into two parts: (i) $\rho _{(\alpha (\log ),\beta ,\gamma
)}[f]=\rho _{(\alpha ,\beta ,\gamma )}[A_{0}]$, (ii) $\mu _{(\alpha (\log
),\beta ,\gamma )}[f]=\mu _{(\alpha ,\beta ,\gamma )}[A_{0}]$.\newline
(i) First, we prove that $\rho _{1}=\rho _{(\alpha (\log ),\beta ,\gamma
)}[f]\geq \rho _{(\alpha ,\beta ,\gamma )}[A_{0}]=\rho _{0}$. \ Suppose the
contrary $\rho _{1}=\rho _{(\alpha (\log ),\beta ,\gamma )}[f]<\rho
_{(\alpha ,\beta ,\gamma )}[A_{0}]=\rho _{0}.$ From $\left( \ref{1.1}\right)
,$ we can write%
\begin{equation}
A_{0}\left( z\right) =-\left( \frac{f^{\left( k\right) }}{f}+A_{k-1}\left(
z\right) \frac{f^{\left( k-1\right) }}{f}+\cdots +A_{1}\left( z\right) \frac{%
f^{\prime }}{f}\right) .  \label{4.19}
\end{equation}%
By Lemma \ref{l3.6} and $\left( \ref{4.19}\right) $, we have%
\begin{equation*}
m\left( r,A_{0}\right) \leq \underset{j=1}{\overset{k-1}{\sum }}m\left(
r,A_{j}\right) +\underset{j=1}{\overset{k}{\sum }}m\left( r,\frac{f^{\left(
j\right) }}{f}\right) +\log k
\end{equation*}%
\begin{equation}
\leq \underset{j=1}{\overset{k-1}{\sum }}m\left( r,A_{j}\right) +O\left(
\exp \left\{ \alpha ^{-1}\left( \left( \rho _{1}+\frac{\varepsilon }{2}%
\right) \beta \left( \log \gamma \left( r\right) \right) \right) \right\}
\right)   \label{4.20}
\end{equation}%
holds possibly outside of an exceptional set $E_{5}\subset \left( 0,+\infty
\right) $ with finite linear measure. Suppose that%
\begin{equation*}
\underset{r\rightarrow +\infty }{\lim \sup }\frac{\overset{k-1}{%
\sum\limits_{j=1}}m\left( r,A_{j}\right) }{m\left( r,A_{0}\right) }=\sigma
<\kappa <1.
\end{equation*}%
Then for sufficiently large $r$, we have%
\begin{equation}
\overset{k-1}{\sum\limits_{j=1}}m\left( r,A_{j}\right) <\kappa m\left(
r,A_{0}\right) .  \label{4.21}
\end{equation}%
By $\left( \ref{4.20}\right) $ and $\left( \ref{4.21}\right) $, we have%
\begin{equation*}
\left( 1-\kappa \right) m\left( r,A_{0}\right) \leq O\left( \exp \left\{
\alpha ^{-1}\left( \left( \rho _{1}+\frac{\varepsilon }{2}\right) \beta
\left( \log \gamma \left( r\right) \right) \right) \right\} \right) ,\text{ }%
r\notin E_{5}.
\end{equation*}%
It follows that
\begin{equation}
T\left( r,A_{0}\right) =m\left( r,A_{0}\right) \leq \exp \left\{ \alpha
^{-1}\left( \left( \rho _{1}+\varepsilon \right) \beta \left( \log \gamma
\left( r\right) \right) \right) \right\} ,\text{ }r\notin E_{5}.
\label{4.22}
\end{equation}%
Hence%
\begin{equation*}
\frac{\alpha \left( \log T\left( r,A_{0}\right) \right) }{\beta \left( \log
\gamma \left( r\right) \right) }\leq \rho _{1}+\varepsilon
\end{equation*}%
and
\begin{equation*}
\rho _{(\alpha ,\beta ,\gamma )}[A_{0}]=\underset{r\longrightarrow +\infty }{%
\lim \sup }\frac{\alpha \left( \log T\left( r,A_{0}\right) \right) }{\beta
\left( \log \gamma \left( r\right) \right) }\leq \rho _{1}+\varepsilon .
\end{equation*}%
Since $\varepsilon >0$ is arbitrary, then we obtain $\rho _{(\alpha ,\beta
,\gamma )}[A_{0}]\leq \rho _{1}.$ This contradiction proves the inequality $%
\rho _{(\alpha (\log ),\beta ,\gamma )}[f]\geq \rho _{(\alpha ,\beta ,\gamma
)}[A_{0}]$. On the other hand, by Lemma \ref{l3.7}, we have
\begin{equation}
\rho _{(\alpha (\log ),\beta ,\gamma )}[f]\leq \max \{\rho _{(\alpha ,\beta
,\gamma )}[A_{j}]:j=0,1,...,k-1\}=\rho _{(\alpha ,\beta ,\gamma )}[A_{0}].
\label{4.23}
\end{equation}%
Hence every solution $f\not\equiv 0$ of equation\textit{\ }$\left( \ref{1.1}%
\right) \ $satisfies $\rho _{(\alpha (\log ),\beta ,\gamma )}[f]=\rho
_{(\alpha ,\beta ,\gamma )}[A_{0}].$\newline
(ii) By using similar arguments as in the proofs of Theorem \ref{t2.3}, we
obtain $\mu _{(\alpha (\log ),\beta ,\gamma )}[f]=\mu _{(\alpha ,\beta
,\gamma )}[A_{0}].$ Hence, every solution $f\not\equiv 0$ of equation\textit{%
\ }$\left( \ref{1.1}\right) \ $satisfies
\begin{equation*}
\mu _{(\alpha ,\beta ,\gamma )}[A_{0}]=\mu _{(\alpha (\log ),\beta ,\gamma
)}[f]\leq \rho _{(\alpha (\log ),\beta ,\gamma )}[f]=\rho _{(\alpha ,\beta
,\gamma )}[A_{0}].
\end{equation*}%
The second part of the proof of Theorem \ref{t2.3} completes the proof of
Theorem \ref{t2.5}.

\textbf{Proof of Theorem \ref{t2.6}}. By Lemma \ref{l3.14}, we obtain that
every linearly independent solution of $\left( \ref{1.1}\right) $ satisfies $%
\underset{r\rightarrow +\infty }{\lim \sup }\frac{\log T\left( r,f\right) }{%
m\left( r,A_{0}\right) }>0,$ $r\notin E$. So, every solution $f$\ $\left(
\not\equiv 0\right) $ of $\left( \ref{1.1}\right) $ satisfies $\underset{%
r\rightarrow +\infty }{\lim \sup }\frac{\log T\left( r,f\right) }{m\left(
r,A_{0}\right) }>0,$ $r\notin E$. Hence, there exist $\delta >0$ and a
sequence $\left\{ r_{n}\right\} _{n=1}^{+\infty }$ tending to $\infty $ such
that for sufficiently large $r_{n}\notin E$ and for every solution $f$\ $%
\left( \not\equiv 0\right) $ of $\left( \ref{1.1}\right) $, we have%
\begin{equation}
\log T\left( r_{n},f\right) >\delta m\left( r_{n},A_{0}\right) .
\label{4.24}
\end{equation}%
Since $\mu _{(\alpha ,\beta ,\gamma )}[A_{0}]=\rho _{(\alpha ,\beta ,\gamma
)}[A_{0}],$ then by $\left( \ref{4.24}\right) $, for any given $\varepsilon
>0$ and sufficiently large $r_{n}\notin E,$ we get%
\begin{equation*}
\log T\left( r_{n},f\right) >\delta \exp \left\{ \alpha ^{-1}\left( \left(
\mu _{(\alpha ,\beta ,\gamma )}[A_{0}]-\frac{\varepsilon }{2}\right) \beta
\left( \log \gamma \left( r_{n}\right) \right) \right) \right\}
\end{equation*}%
\begin{equation*}
\geq \exp \left\{ \alpha ^{-1}\left( \left( \mu _{(\alpha ,\beta ,\gamma
)}[A_{0}]-\varepsilon \right) \beta \left( \log \gamma \left( r_{n}\right)
\right) \right) \right\} ,
\end{equation*}%
which implies
\begin{equation}
\rho _{(\alpha (\log ),\beta ,\gamma )}[f]\geq \mu _{(\alpha ,\beta ,\gamma
)}[A_{0}]=\rho _{(\alpha ,\beta ,\gamma )}[A_{0}].  \label{4.25}
\end{equation}%
On the other hand, by Lemma \ref{l3.7}, we have
\begin{equation*}
\rho _{(\alpha (\log ),\beta ,\gamma )}[f]\leq \max \{\rho _{(\alpha ,\beta
,\gamma )}[A_{j}]:j=0,1,...,k-1\}
\end{equation*}%
\begin{equation}
=\mu _{(\alpha ,\beta ,\gamma )}[A_{0}]=\rho _{(\alpha ,\beta ,\gamma
)}[A_{0}].  \label{4.26}
\end{equation}%
By $\left( \ref{4.25}\right) $ and $\left( \ref{4.26}\right) $, we obtain $%
\rho _{(\alpha (\log ),\beta ,\gamma )}[f]=\mu _{(\alpha ,\beta ,\gamma
)}[A_{0}]=\rho _{(\alpha ,\beta ,\gamma )}[A_{0}]$.\newline

\noindent \qquad The second part of the proof of Theorem \ref{t2.3}
completes the proof of Theorem \ref{t2.6}.


\begin{thebibliography}{99}
\bibitem{b1} B. Bela\"{\i}di and S. Hamouda, \textit{Orders of solutions of
an n-th order linear differential equation with entire coefficients}.
Electron. J. Differential Equations 2001, No. 61, 5 pp.

\bibitem{b2} B. Bela\"{\i}di, \textit{Estimation of the hyper-order of
entire solutions of complex linear ordinary differential equations whose
coefficients are entire functions}. Electron. J. Qual. Theory Differ. Equ.
2002, no. 5, 8 pp. DOI: https://doi.org/10.14232/ejqtde.2002.1.5

\bibitem{b3} B. Bela\"{\i}di, \textit{Growth and oscillation of solutions to
linear differential equations with entire coefficients having the same order}%
. Electron. J. Differential Equations 2009, No. 70, 10 pp.

\bibitem{b4} B. Bela\"{\i}di, \textit{Growth of }$\rho _{\varphi }$\textit{%
-order solutions of linear differential equations with entire coefficients}.
PanAmer. Math. J. 27 (2017), no. 4, 26--42.

\bibitem{b5} B. Bela\"{\i}di, \textit{Fast growing solutions to linear
differential equations with entire coefficients having the same }$\rho
_{\varphi }$\textit{-order}. J. Math. Appl. 42 (2019), 63--77. DOI:
10.7862/rf.2019.4

\bibitem{b6} B. Bela\"{\i}di and T. Biswas, \textit{Study of complex
oscillation of solutions of a second order linear differential equation with
entire coefficients of }$(\alpha ,\beta ,\gamma )$\textit{-order}. WSEAS
Trans. Math., 21 (2022), 361--370. DOI: 10.37394/23206.2022.21.43.

\bibitem{b7} B. Bela\"{\i}di and T. Biswas, \textit{Growth properties of
solutions of complex differential equations with entire coefficients of
finite }$(\alpha ,\beta ,\gamma )$\textit{-order}. Electron. J. Differential
Equations, 2023, No. 27, 14 pp.

\bibitem{b8} B. Bela\"{\i}di and T. Biswas, \textit{Growth of }$(\alpha
,\beta ,\gamma )$\textit{-order solutions of linear differential equations
with entire coefficients}. Accepted in Novi Sad Journal of Mathematics. DOI:
10.30755/NSJOM.16382

\bibitem{b} L. G. Bernal, \textit{On growth }$k$\textit{-order of solutions
of a complex homogeneous linear differential equations}. Proc. Amer. Math.
Soc. 101(1987), no 2, 317--322. DOI: https://doi.org/10.2307/2046002

\bibitem{tb1} T. Biswas and C. Biswas, \textit{Some inequalities relating to
generalized relative order }$(\alpha ,\beta )$\textit{\ and generalized
relative type }$(\alpha ,\beta )$\textit{\ of meromorphic functions with
respect to entire function}. J. Ramanujan Soc. Math. Math. Sci. 11 (2023),
no. 1, 1--16.

\bibitem{tb2} T. Biswas and C. Biswas, \textit{Generalized relative order }$%
(\alpha ,\beta )$\textit{\ oriented some growth analysis of composite }$p$%
\textit{-adic entire functions}. Palest. J. Math. 12 (2023), no. 3, 432--442.

\bibitem{bou} R. Bouabdelli and B. Bela\"{\i}di, \textit{Growth and complex
oscillation of linear differential equations with meromorphic coefficients
of }$\left[ p,q\right] $\textit{-}$\varphi $\textit{\ order. }Int. J. Anal.
Appl. 6 (2014), No. 2, 178--194.

\bibitem{c1} Z. X. Chen and C. C. Yang, \textit{Quantitative estimations on
the zeros and growths of entire solutions of linear differential equations}.
Complex Variables Theory Appl. 42 (2000), no. 2, 119--133.
https://doi.org/10.1080/17476930008815277

\bibitem{CS} I. Chyzhykov and N. Semochko, \textit{Fast growing entire
solutions of linear differential equations}. Math. Bull. Shevchenko Sci.
Soc. 13 (2016), 68--83.

\bibitem{gs} G. G. Gundersen, \textit{Estimates for the logarithmic
derivative of a meromorphic function, plus similar estimates}. J. London
Math. Soc. (2) 37 (1988), no. 1, 88--104.
https://doi.org/10.1112/jlms/s2-37.121.88

\bibitem{gg} G. G. Gundersen, \textit{Finite order solutions of second order
linear differential equations}. Trans. Amer. Math. Soc. 305 (1988), no. 1,
415--429.{\small \ }https://doi.org/10.1090/S0002-9947-1988-0920167-5

\bibitem{go} A. Goldberg and I. Ostrovskii, \textit{Value distribution of
meromorphic functions}. Transl. Math. Monogr., vol. 236, Amer. Math. Soc.,
Providence RI, 2008.

\bibitem{1} W. K. Hayman, \textit{Meromorphic functions}. Oxford
Mathematical Monographs, Clarendon Press, Oxford 1964.

\bibitem{17} W. K. Hayman, \textit{The local growth of power series: a
survey of the Wiman-Valiron method}. Canad. Math. Bull. 17 (1974), no. 3,
317--358. DOI: https://doi.org/10.4153/CMB-1974-064-0

\bibitem{h} Y. Z. He and X. Z. Xiao, \textit{Algebroid functions and
ordinary differential equations}. Science Press, Beijing, 1988 (in Chinese).

\bibitem{20} J. Heittokangas, R. Korhonen and J. R\"{a}tty\"{a}, \textit{%
Growth estimates for solutions of linear complex differential equations}.
Ann. Acad. Sci. Fenn. Math. 29 (2004), no. 1, 233--246.

\bibitem{HWWY} J. Heittokangas, J. Wang, Z. T. Wen and H. Yu, \textit{%
Meromorphic functions of finite }$\varphi $\textit{-order and linear }$q$%
\textit{-difference equations}. J. Difference Equ. Appl. 27 (2021), no. 9,
1280--1309. DOI: 10.1080/10236198.2021.1982919.

\bibitem{gl} G. Jank and L. Volkmann, \textit{Einfuhrung in die Theorie der
Ganzen und Meromorphen Funktionen mit Anwendungen auf Differentialgleichungen%
}. Birkh\"{a}user Verlag, Basel, 1985.

\bibitem{8} O. P. Juneja, G. P. Kapoor and S. K. Bajpai, \textit{On the }$%
(p,q)$\textit{-order and lower }$(p,q)$\textit{-order of an entire function}%
. J. Reine Angew. Math. 282 (1976), 53--67.

\bibitem{9} O. P. Juneja, G. P. Kapoor and S. K. Bajpai, \textit{On the }$%
(p,q)$\textit{-type and lower }$(p,q)$\textit{-type of an entire function}.
J. Reine Angew. Math. 290 (1977), 180--190.

\bibitem{k} M. A. Kara and B. Bela\"{\i}di, \textit{Growth of }$\varphi $%
\textit{-order solutions of linear differential equations with meromorphic
coefficients on the complex plane}. Ural Math. J. 6 (2020), no. 1, 95--113.
DOI: http://dx.doi.org/10.15826/umj.2020.1.008

\bibitem{k1} M. A. Kara and B. Bela\"{\i}di, \textit{Fast growth of the
logarithmic derivative with applications to complex differential equations}.
Sarajevo J. Math. 19(32) (2023), no. 1, 63--77. DOI:10.5644/SJM.19.01.05

\bibitem{13} L. Kinnunen, \textit{Linear differential equations with
solutions of finite iterated order}. Southeast Asian Bull. Math. 22 (1998),
no. 4, 385--405.

\bibitem{19} I. Laine, \textit{Nevanlinna theory and complex differential
equations}. De Gruyter Studies in Mathematics, 15. Walter de Gruyter \& Co.,
Berlin, 1993. https://doi.org/10.1515/9783110863147\

\bibitem{10} L. M. Li and T. B. Cao, \textit{Solutions for linear
differential equations with meromorphic coefficients of }$(p,q)$\textit{%
-order in the plane}. Electron. J. Differential Equations 2012, No. 195, 15
pp.

\bibitem{11} J. Liu, J. Tu and L. Z. Shi, \textit{Linear differential
equations with entire coefficients of }$[p,q]$\textit{-order in the complex
plane}. J. Math. Anal. Appl. 372 (2010), no. 1, 55--67.
https://doi.org/10.1016/j.jmaa.2010.05.014

\bibitem{l} S. G. Liu, J. Tu and H. Zhang, \textit{The growth and zeros of
linear differential equations with entire coefficients of }$[p,q]$\textit{-}$%
\varphi $\textit{\ order}. J. Comput. Anal. Appl., 27 (2019), no. 4, 681-689.

\bibitem{L} J. Long, H. Qin and L. Tao, \textit{On }$[p,q]_{,\varphi }$%
\textit{-order and complex differential equations}. J. Nonlinear Math. Phys.
(2023). https://doi.org/10.1007/s44198-023-00107-7

\bibitem{MST} O. M. Mulyava, M. M. Sheremeta and Yu. S. Trukhan, \textit{%
Properties of solutions of a heterogeneous differential equation of the
second order}. Carpathian Math. Publ. 11 (2019), no. 2, 379--398.
https://doi.org/10.15330/cmp.11.2.379-398

\bibitem{DS} D. Sato, \textit{On the rate of growth of entire functions of
fast growth}. Bull. Amer. Math. Soc. 69 (1963), 411--414.

\bibitem{SC} A. Sch\"{o}nhage, \textit{\"{U}ber das Wachstum
zusammengesetzter Funktionen}. Math. Z. 73 (1960), 22--44.
https://doi.org/10.1007/BF01163267

\bibitem{S} N. Semochko, \textit{On solutions of linear differential
equations of arbitrary fast growth in the unit disc}. Mat. Stud. 45 (2016),
no. 1, 3--11. doi:10.15330/ms.45.1.3-11

\bibitem{STX} X. Shen, J. Tu and H. Y. Xu, \textit{Complex oscillation of a
second-order linear differential equation with entire coefficients of }$%
[p,q] $\textit{-}$\varphi $\textit{\ order}. Adv. Difference Equ. 2014,
2014:200, 14pp.

\bibitem{MNS} M. N. Sheremeta, \textit{Connection between the growth of the
maximum of the modulus of an entire function and the moduli of the
coefficients of its power series expansion}. Izv. Vyssh. Uchebn. Zaved.
Mat., 2 (1967), 100--108. (in Russian).

\bibitem{t2} J. Tu and C.-F. Yi, \textit{On the growth of solutions of a
class of higher order linear differential equations with coefficients having
the same order}. J. Math. Anal. Appl. 340 (2008), no. 1, 487--497.
https://doi.org/10.1016/j.jmaa.2007.08.041

\bibitem{t3} J. Tu and Z.-X. Chen, \textit{Growth of solutions of complex
differential equations with meromorphic coefficients of finite iterated order%
}. Southeast Asian Bull. Math. 33 (2009), no. 1, 153--164.

\bibitem{v} G. Valiron, \textit{Lectures on the general theory of integral
functions}{\small . }translated by E. F. Collingwood, Chelsea Publishing
Company, New York, 1949.

\bibitem{w1} H. Wittich, \textit{Zur Theorie linearer
Differentialgleichungen im Komplexen}, Ann. Acad. Sci. Fenn. Ser. A \textbf{%
I 379} (1966).

\bibitem{w2} H. Wittich, \ \textit{Neuere Untersuchungen }$\overset{..}{%
\text{\textit{u}}}$\textit{ber eindeutige analytishe Funktionen}, \ 2nd
Edition, Springer-Verlag, \ Berlin-Heidelberg-New York, 1968.

\bibitem{15} H. Y. Xu and J. Tu, \textit{Oscillation of meromorphic
solutions to linear differential equations with coefficients of }$[p,q]$%
\textit{-order}. Electron. J. Differential Equations 2014, No. 73, 14 pp.

\bibitem{3} C. C. Yang and H. X. Yi, \textit{Uniqueness theory of
meromorphic functions}. Mathematics and its Applications, 557. Kluwer
Academic Publishers Group, Dordrecht, 2003.
\end{thebibliography}
\end{document}